\documentclass{amsart}
\usepackage{amsmath}
  \usepackage{paralist}
  \usepackage{graphics} 
  \usepackage{epsfig} 
\usepackage{graphicx}  \usepackage{epstopdf}
\usepackage{amsmath,mathtools,amssymb,amsthm,amsfonts,bm,latexsym,graphicx,float, caption, subcaption,yfonts,cite,xfrac}
\graphicspath{{Figures/}}
\def\R{\Bbb{R}}
\def\Z{\Bbb{Z}}
\def\M{\mathcal{M}}
\def\S{\mathbb{S}}

\everymath{\displaystyle}
 \usepackage[colorlinks=true]{hyperref}
\hypersetup{urlcolor=blue, citecolor=red}

\newtheorem{theorem}{Theorem}[section]

\newtheorem{prop}[theorem]{Proposition}

\theoremstyle{definition}
\newtheorem{definition}[theorem]{Definition}
\newtheorem{remark}[theorem]{Remark}

\newcount\theTime
\newcount\theHour
\newcount\theMinute
\newcount\theMinuteTens
\newcount\theScratch
\theTime=\number\time
\theHour=\theTime
\divide\theHour by 60
\theScratch=\theHour
\multiply\theScratch by 60
\theMinute=\theTime
\advance\theMinute by -\theScratch
\theMinuteTens=\theMinute
\divide\theMinuteTens by 10
\theScratch=\theMinuteTens
\multiply\theScratch by 10
\advance\theMinute by -\theScratch

\def\today{{\number\day\space
 \ifcase\month\or
  January\or February\or March\or April\or May\or June\or
  July\or August\or September\or October\or November\or December\fi
 \space\number\year}}

\title[Weighted Divergent Beam and Cone Transforms] 
      {Inversion of Weighted Divergent Beam and Cone Transforms}

\author[Peter Kuchment and Fatma Terzioglu]{}


 \email{kuchment@math.tamu.edu}
 \email{fatma@math.tamu.edu}

\thanks{This work was supported in part by NSF DMS grant 1211463.}

\thanks{$^*$ Corresponding author: Fatma Terzioglu, e-mail: fatma@math.tamu.edu}

\begin{document}
\maketitle

\centerline{\scshape Peter Kuchment}
\medskip
{\footnotesize
 \centerline{ Department of Mathematics}
   \centerline{Texas A$\&$M University}
   \centerline{College Station, TX 77843-3368, USA}
} 

\medskip

\centerline{\scshape Fatma Terzioglu$^*$}
\medskip
{\footnotesize
 \centerline{ Department of Mathematics}
   \centerline{Texas A$\&$M University}
   \centerline{College Station, TX 77843-3368, USA}
}

\bigskip


\begin{abstract}
In this paper, we investigate the relations between the Radon and weighted divergent beam and cone transforms. Novel inversion formulas are derived for the latter two. The weighted cone transform arises, for instance, in image reconstruction from the data obtained by Compton cameras, which have promising applications in various fields, including biomedical and homeland security imaging and gamma ray astronomy. The inversion formulas are applicable for a wide variety of detector geometries in any dimension. The results of numerical implementation of some of the formulas in dimensions two and three are also provided.
\end{abstract}

\section{Introduction}
In this paper, we focus mainly on analytic and numerical inversion of an integral transform (\emph{cone} or \emph{Compton} transform) that maps a function to its integrals over conical surfaces with a weight equal to some power of the distance from the cone's vertex. It arises in various imaging techniques, most prominently, in modeling of the data provided by the so-called Compton camera, which has novel applications in various fields including medical and industrial imaging, homeland security, and gamma ray astronomy \cite{Basko,Todd,Singh,ADHKK,ACCHKOR,XMCK}. In Compton camera setting, the vertices of the cones correspond to the locations of the detection sites on the scattering detector. More information on the working principle of a Compton camera can be found, for example, in \cite{ADHKK,Basko,Terzioglu,Todd,Singh}.

It has been mentioned in various papers, e.g. \cite{Basko,Smith,Maxim} that, depending upon the engineering of the detector, various power weights can appear in the surface integral. However, more work needs to be done to determine the weight factor that accurately represents the projections obtained from a Compton camera. Several works, e.g. \cite[and references therein]{Hristova,Hristova2015,Amb2,Amb3,Basko,Smith,Cree,Gouia-Zarrad-Ambarts,Truong,NgTr2005,Terzioglu,KuchTer,Florescu} concentrated on the case of pure surface measure on the cone. Here, we consider a weight that is equal to some power of the distance to the vertex (detection site). An alternative inversion formula for such transform that assumes the vertices of the cones are located on a given straight line is provided in \cite{Moon}. A reconstruction formula for such transform defined on the cones having vertices on a hyperplane and a central axis orthogonal to this hyperplane is derived in \cite{Haltmeier}. In comparison, the formulas we derive allow for a wide variety of cone vertex (a.k.a. detector or source) geometries, which do not allow for harmonic analysis, but satisfy what we call in this paper Tuy's condition (Definition \ref{D:Tuy}). 

A closely intertwined with the weighted cone transform is what is called \emph{weighted divergent beam transform}, which integrates a function over rays with a weight equal to some power of the distance to the starting point (source) of the ray. We thus study it in some details, which leads eventually to the desired weighted cone transform inversions. When the weight factor is not present, this is the well studied and important for the $3D$ X-ray CT divergent (or cone) beam transform (see e.g. \cite[and references therein]{Grangeat, Hamaker, SmithConeBeam, Gelfand-Goncharov, Tuy, Katsevich2002, Katsevich2004, NattWubb}).

In order to avoid being distracted from the main purpose of this text, we assume that the functions in question belong to the Schwartz space $\mathcal{S}$ of smooth fast decaying functions. This allows us to skip discussions of applicability of various transforms. However, as in the case of Radon transform (see, e.g. \cite{Natt_old,KuchCBMS}), the results have a much wider area of applicability, since the derived formulas can be extended by continuity (although we do not do this in the current text) to some wider functional spaces. This is confirmed, in particular, by our successful numerical implementations for discontinuous (piecewise continuous) phantoms. The issues of appropriate functional spaces will be addressed elsewhere. 

We also adopt the standard abuse of notations, writing the action of a distribution $T$ on a test function $\varphi$, $\langle T,\varphi \rangle$, as $\textstyle\int T(x)\varphi(x)dx$.

The paper is organized as follows. In section \ref{S:Weighted}, we define the weighted divergent beam and cone transforms, and describe a simple relation between them. In section \ref{S:InvWeighted}, we present a variety of inversion formulas for the weighted divergent beam transform (Theorems \ref{T:D-sub} and \ref{T:D-sub-even}). We then derive another integral relation between the weighted divergent beam and cone transforms, which leads to new inversion formulas for the $n$-dimensional weighted cone transform (Theorem \ref{T:C-sub}). In section \ref{S:InvRadon}, we investigate the relation between the Radon and weighted divergent beam and cone transforms. This enables us to derive other inversion formulas for the latter two (Theorem \ref{T:InviaRadon}). Section \ref{S:numerics} contains the results of numerical implementation of some of the inversion formulas for the weighted cone transform in dimensions two and three for two different vertex geometries, as well as examples of numerical inversion of two weighted divergent beam transforms in dimension three. Conclusions and remarks can be found in section \ref{S:remarks}, followed by the acknowledgments section.

\section{Weighted Divergent Beam and Cone Transforms}\label{S:Weighted}
In this section, we define the closely related weighted divergent beam and cone transforms.
 \begin{definition}
 For $k>-1$, the \emph{$k$-weighted divergent beam transform} of a function $f \in \mathcal{S}(\R^n)$ is defined by
 \begin{align}\label{Div}
D^k f(u,\sigma) = D_u^kf(\sigma): = \int_0^\infty f(u+\rho \sigma) \rho^k d\rho,
\end{align}
where $u\in\R^n$ is the \emph{source} of the beam $\{u+\rho \sigma\}|_{ \rho \geq 0}$ and $\sigma\in\S^{n-1}$ is the unit vector in the direction of the beam.
 \end{definition}

Consider now a circular cone\footnote{The word ``cone'' in this paper always means a surface, rather than solid cone.} $\mathfrak{S}$ in $\mathbb{R}^n$. Its surface can be parametrized by a triple $(u,\beta,\psi)$, where $u \in \mathbb{R}^n$ is the cone's \emph{vertex (apex)}\footnote{In the Compton camera imaging, cone's vertex corresponds to a detection location.}, the unit vector $\beta \in \S^{n-1}$ is directed toward cone's interior along the cone's axis, and $\psi \in (0,\pi)$ is the opening angle (see Fig. \ref{fig:camera&scatter}).
 \begin{figure}[ht]
\begin{center}
        \includegraphics[width=0.3\textwidth]{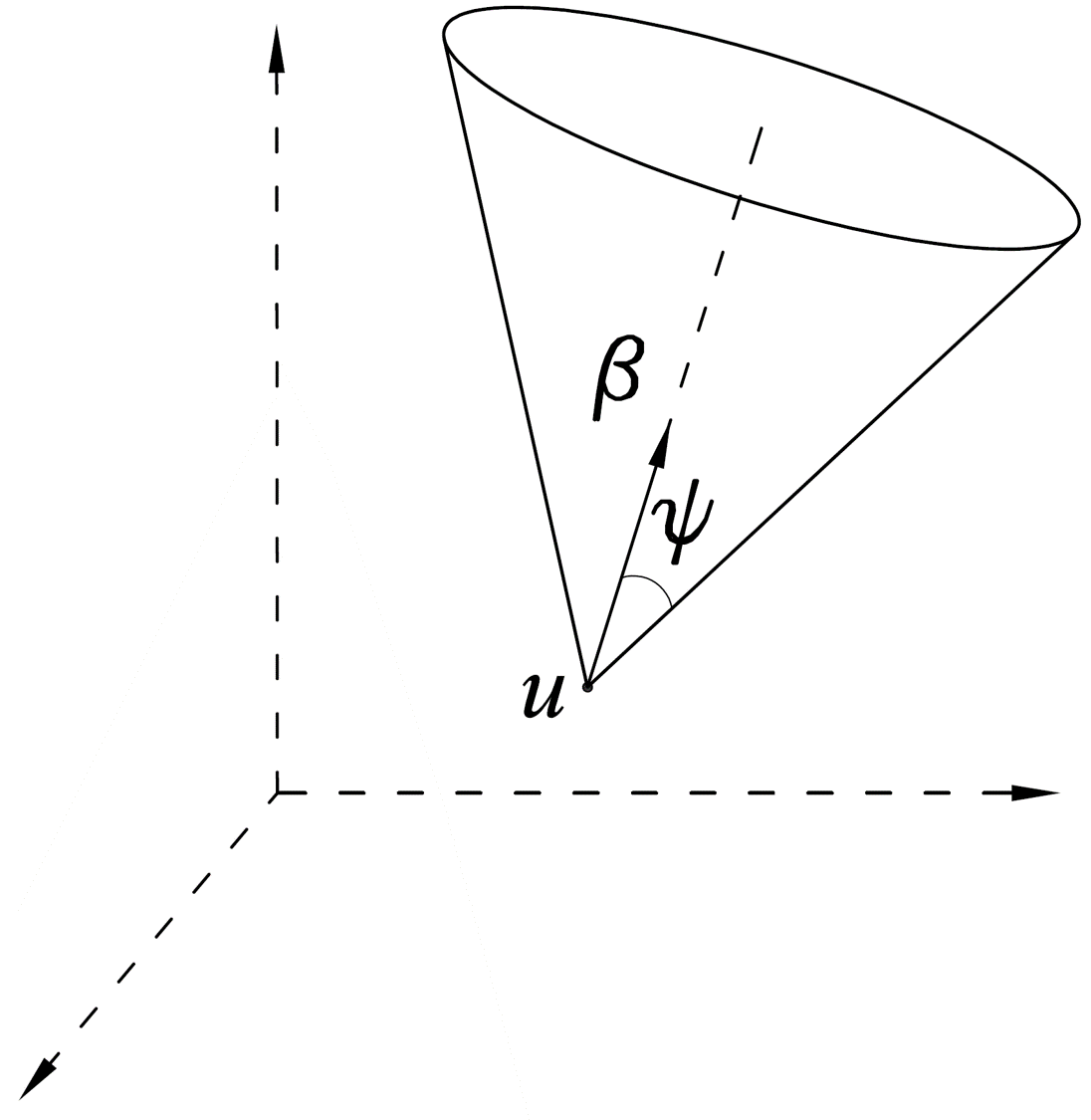}
        \caption{A cone with vertex $u \in \R^n$, central axis direction vector $\beta \in \S^{n-1}$ and opening angle $\psi \in (0,\pi)$.}
       \label{fig:camera&scatter}
\end{center}
\end{figure}
A point  $x \in \R^n$ lies on $\mathfrak{S}(u,\beta,\psi)$ iff $(x-u)\cdot\beta=|x-u| \cos \psi.$
\begin{definition}
Let $k \in \mathbb{Z}_{+}=\{0,1,2,...\}$\footnote{At this step, one can allow all real values $k>-1$, while later on in sections \ref{S:InvWeighted} and  \ref{S:InvRadon}, $k \in \mathbb{Z}_{+}:=\{0,1,2,\dots\}$ will be important.}, and suppose that $f \in \mathcal{S}(\R^n)$. We define \emph{the $k$-weighted cone transform} $C^k$ of $f$ as
\begin{align}\label{cone trans nodelta}
C^k f(u,\beta,\psi): = \int_{\mathfrak{S}(u,\beta,\psi)}f(x)|x-u|^{k-n+2} dS(x),
\end{align}
where $dS$ is the surface measure on the cone $\mathfrak{S}$.
In other words,
\begin{align} \label{cone trans delta}
 C^k f(u,\beta,\psi) =  \sin \psi \int_{\R^n} f(x) \delta((x-u)\cdot\beta-|x-u| \cos \psi)|x-u|^{k-n+2} dx,
 \end{align}
 where $dx$ is the Lebesgue measure on $\R^n$.
 \end{definition}

 \begin{remark}
 We note that $k=n-2$ corresponds to the case of pure surface measure on the cone.
 \end{remark}

\begin{remark}
\textbf{We will say just ``weighted" cone or divergent beam transform}, when no confusion about the value of $k$ can arise.\end{remark}

Changing variables in (\ref{cone trans delta}) as $x=u+\rho \sigma$ for $\rho \in [0,\infty)$ and $\sigma \in \S^{n-1}$, and using the fact that $\delta$ is homogeneous of degree $-1$, we make the following simple observation:
\begin{prop}
\begin{align}\label{ConeFan}
 C^k f(u,\beta,\psi) &=  \sin \psi\int_{\S^{n-1}} \int_0^\infty f(u+\rho \sigma)\rho^k d \rho \; \delta(\sigma \cdot\beta-\cos \psi) d \sigma \nonumber\\
 & = \sin \psi \int_{\S^{n-1}} D_u^kf(\sigma) \delta(\sigma \cdot\beta-\cos \psi) d \sigma.
 \end{align}

By letting $t=\cos \psi$, we can rewrite (with an abuse of notation) $C^kf$ as
\begin{align}\label{Ct}
C^k f(u,\beta,t):=
\begin{cases}
   \sqrt{1-t^2} \int_{\S^{n-1}} D_u^kf(\sigma) \delta(\sigma \cdot\beta-t) d \sigma,  &  \mbox{if } |t| \leq 1 \\
    0, & \mbox{ otherwise}.\\
\end{cases}
\end{align}
\end{prop}
\section{Inversion of Weighted Divergent Beam and Cone Transforms}\label{S:InvWeighted}
In this section, we present a variety of inversion formulas for the weighted divergent beam transform. We then derive another integral relation between the weighted divergent beam and cone transforms, which enables us to develop new inversion formulas for the $n$-dimensional weighted cone transform.
\subsection{Inversion of the Weighted Divergent Beam Transform}
If $f \in \mathcal{S}(\R^n)$, for each $u \in \R^n$, $D_u^k f(\sigma)$ can be uniquely extended to a smooth function on $\R^n \backslash \{0\}$ homogeneous of degree $-(k+1)$:
\begin{align}\label{Dext}
D_u^k f(x) = \frac{1}{|x|^{k+1}} D_u^k f(\frac{x}{|x|}).
\end{align}
This function is locally integrable with respect to $x\in\R^n$, provided  $k<n-1$, and has a well-defined Fourier transform as a tempered distribution (see e.g. \cite{Gelfand}), i.e., for each $\varphi \in \mathcal{S}({\R^n})$,
\begin{align}\label{FDext}
\langle \widehat{D_u^kf}(\xi), \varphi (\xi) \rangle = \int_{\R^n} D_u^k f(y) \hat{\varphi}(y) dy.
\end{align}

In the following, we derive inversion formulas for the divergent beam transform that are analogs of the well known \cite{Tuy,NattWubb} Tuy's inversion formula, which addresses the case when $k=0$ in dimension three, and the sources (detectors in the Compton camera case) move along a curve.

\begin{definition}In the rest of the paper, the shorthand notations $\partial_{u_j}$ and $\partial_u$ will be used for the partial derivatives $\partial/\partial u_j$ and gradient $\nabla_u$ with respect to the variables $u$.
\end{definition}

\begin{theorem}\label{T:IDT}
Let $f \in \mathcal{S}(\R^n)$, and all source locations $u$ are accessible. Then,
\begin{align}\label{IDT}
f(x)= \frac{(-i)^{k+1}}{(2\pi)^n} \int_{\S^{n-1}} \left(\Delta_u^{(k+1)/2} \widehat{D_u^kf}(\theta)\right)\big|_{u=x} d\theta,
\end{align}
where $\Delta_u:=\textstyle\sum_j \partial_{u_j}^2$ is the Laplace operator with respect to the variable $u$, and its power when $k$ is not an odd integer is understood as the corresponding Riesz potential (see, e.g. \cite{Natt_old}).
\end{theorem}
\begin{proof}
 Let $f \in \mathcal{S}(\R^n)$. For any $\theta \in \S^{n-1}$, the Fourier transform of $D_u^kf$ satisfies
\begin{align}\label{DFT}
\widehat{D_u^kf}(\theta)= \int_0^\infty e^{i\rho \theta \cdot u} \hat{f}(\rho \theta) \rho^{n-k-2} d \rho.
\end{align}
Indeed, for any $\varphi \in \mathcal{S}({\R^n})$, due to $D_u^kf$ being homogeneous of degree $-(k+1)$ and changing to polar variables $y=s\omega$, we have
\begin{align*}
\langle \widehat{D_u^kf}, \varphi \rangle
& = \langle D_u^kf, \hat{\varphi} \rangle  = \int_{\R^n} D_u^k f(y) \hat{\varphi}(y) dy\\
 &= \int_{\S^{n-1}} \int_0^\infty  D_u^k f(\omega) \hat{\varphi}(s\omega) s^{n-k-2}dsd\omega \\
& = \int_{\R^n} \varphi(x) \int_{\S^{n-1}} \int_0^\infty   \int_0^\infty e^{-i x \cdot s\omega}  f(u+r \omega) s^{n-k-2}ds \; r^k dr d\omega  \; dx.
\end{align*}
Now, changing variables in $s$ to $\rho = s/r$, and then letting $y=u+r\omega$, we get
\begin{align*}
\langle \widehat{D_u^kf}, \varphi \rangle
&=\int_{\R^n} \varphi(x) \int_0^\infty \int_{\S^{n-1}} \int_0^\infty  e^{-i x \cdot r \rho \omega}   f(u+r \omega) r^{n-1} dr d\omega \rho^{n-k-2} d\rho \; dx\\
&= \int_{\R^n} \varphi(x) \int_0^\infty \int_{\R^n} e^{-i \rho x \cdot (y-u)}   f(y) dy \rho^{n-k-2} d\rho \; dx,
\end{align*}
which implies \eqref{DFT}.

The following simple formula holds for any unit vector $\theta$:
\begin{equation}\label{E:LaplExp}
\Delta_u^{(k+1)/2}e^{i\rho\theta\cdot u}=(i\rho)^{k+1}e^{i\rho\theta\cdot u}.
\end{equation}
Thus, applying (k+1)/2-th power of the Laplace operator with respect to $u$ to \eqref{DFT}, we obtain
\begin{align}\label{E:LaplDBT}
\Delta_u^{(k+1)/2} \widehat{D_u^kf}(\theta)= i^{k+1} \int_0^\infty e^{i\rho \theta \cdot u} \hat{f}(\rho \theta) \rho^{n-1} d \rho.
\end{align}
Now, recalling the Fourier inversion formula in polar coordinates
\begin{align}\label{IFFT}
f(x)= \frac{1}{(2\pi)^{n}} \int_{\S^{n-1}} \int_0^\infty e^{i\rho \theta \cdot x} \hat{f}(\rho \theta) \rho^{n-1} d \rho d\theta
\end{align}
and comparing with \eqref{E:LaplDBT}, we obtain the desired formula
$$f(x)= \frac{(-i)^{k+1}}{(2\pi)^n} \int_{\S^{n-1}} \left(\Delta_u^{(k+1)/2} \widehat{D_u^kf}(\theta)\right)\big|_{u=x} d\theta.$$
\end{proof}
\begin{remark}\label{R:defic_beam}
Considering formula (\ref{IDT}), one realizes quickly that it is not very useful, since it requires ``sources'' $u$ of the beams to be available throughout the whole space. In the Compton camera case, as well as in $3D$ CT, this would require detectors/sources to be placed throughout the object imaged, which is impossible.

Moreover, in this case, one deals with just a deconvolution problem, and a severely overdetermined one at that (the dimension of the data used is $2n-1$ instead of $n$). Thus, there must exist formulas requiring much less data, in particular allowing the detectors $u$ to be situated only outside the object being imaged (e.g. Tuy's formula only requires an arc of external sources).

This is also related to the interesting question about ``admissible'' complexes of cones that provide enough data for stable reconstruction. We have already briefly addressed this issue in \cite{Terzioglu,KuchTer} and plan to have more detailed discussion elsewhere.

Here we show an example of how such deficiency can be alleviated for the weighted divergent beam transform.
\end{remark}
\begin{definition}\label{D:Tuy}\indent
\begin{itemize}
\item Let $\M\subset\R^n$ be a smooth $d$-dimensional submanifold. We will say that it satisfies the \textbf{Tuy's condition} with respect to a subset $V\subset \R^n$, if any hyperplane intersecting $V$ has a non-tangential intersection with $\M$.

Equivalently: for any $x\in V$ and unit vector $\theta \in \S^{n-1}$, there exists a point $u\in\M$ such that $\theta \cdot x = \theta \cdot u$, and $\theta$ is not normal to $\M$ at the point $u$.
\item We denote by $P_u$ the orthogonal projection onto the tangent space to $\M$ at the point $u\in\M$.
\end{itemize}
\end{definition}

\begin{remark}
Notice that the above condition is a strengthened version of what was called \textbf{admissibility condition} in \cite{KuchTer,Terzioglu}.
\end{remark}

\begin{theorem}\label{T:D-sub}
Let $k$ be an odd natural number and $\M\subset\R^n$ satisfies Tuy's condition with respect to a compact $V$. Then, for any homogeneous linear elliptic differential operator $L(u,\partial_u)$ of order $k+1$ on $\M$ and any smooth function $f$ supported in $V$, the following inversion formula holds:
 \begin{align}\label{D-sub}
 f(x)=\frac{1}{(2\pi)^n}\int_{\S^{n-1}}\frac{1}{L(u,P_u\theta)}L(u,\partial_u) \widehat{D_u^kf}(\theta)d\theta,
 \end{align}
 where $u\in\M$ is related to $x$ and $\theta$ as in the Tuy's condition.

 If $\M$ is one-dimensional, then $k$ can be any natural number (in this case, when $k=0$, one ends up with the standard Tuy's formula).
\end{theorem}
\begin{remark}
Notice that $u\in\M$ in (\ref{D-sub}) depends on both $x$ and $\theta$ and that $L(u,P_u\theta)$ does not vanish if the Tuy's condition is satisfied.
\end{remark}
\begin{proof}
The proof follows exactly the one of Theorem \ref{T:IDT}, using at the end the formula for the symbol of a homogeneous differential operator of order $k+1$ (instead of a power of the Laplacian in (\ref{E:LaplExp})):
\begin{equation}\label{E:symbolL}
L(u,\partial_u)e^{i\rho\theta\cdot u}=(i\rho)^{k+1}L(u,P_u\theta)e^{i\rho\theta\cdot u}
\end{equation}
and noticing that the factor $L(u,P_u\theta)$ does not vanish, due to the ellipticity and homogeneity of the operator and the Tuy's condition.
\end{proof}

A serious deficiency in Theorem \ref{T:D-sub} is that, unless $\M$ is one-dimensional, only odd values of $k$ are allowed. This issue can be resolved, paying the price of having a more complex formula.

Indeed, consider the following first order linear differential operator acting tangentially to $\M$, with coefficients depending upon $u\in\M$ and $\theta\in\S^{n-1}$:
\begin{equation}\label{E:OperatorO}
O(u,\theta,\partial_u) := P_u\theta \cdot \partial_u.
\end{equation}
Let also $L(u,\partial_u)$ be an operator like in Theorem \ref{T:D-sub}, but of order $k$.
Applying the composition $O\circ (aL)$ to the exponential $e^{i\rho\theta\cdot u}$, where $a:=1/L(u,P_u\theta)$, and using (\ref{E:symbolL}), we get
\begin{equation}\label{E:compose}
O\circ (aL)e^{i\rho\theta\cdot u}=|P_u\theta|^2(i\rho)^{k+1}e^{i\rho\theta\cdot u}.
\end{equation}
Since the order of the composition is odd, this enables us to extended the inversion formula to the even values of $k$:

\begin{theorem}\label{T:D-sub-even}
Let $k$ be an even natural number and $\M\subset\R^n$ satisfies Tuy's condition with respect to a compact $V$. Then, for any homogeneous elliptic differential operator $L(u,\partial_u)$ of order $k$ on $\M$ and any smooth function $f$ supported in $V$, the following inversion formula holds:
 \begin{align}\label{D-sub-even}
 f(x)=\frac{1}{(2\pi)^n}\int_{\S^{n-1}}\frac{1}{|P_u\theta|^2}\big(O\circ \frac{1}{L(u,P_u\theta)}L\big) \widehat{D_u^kf}(\theta) d\theta,
 \end{align}
 where $u$ is related to $x$ and $\theta$ as in the Tuy's condition.
\end{theorem}
\begin{proof} The proof stays exactly the same, except instead of  (\ref{E:symbolL}), we use (\ref{E:compose}).\end{proof}
\begin{remark}\label{R:manyversions}
Since we are dealing with a severely overdetermined transform, the variety of possible inversion formulas is large and is not exhausted by the ones above. For instance, instead of using the operator
$\textstyle\frac{1}{|P_u\theta|^2}(O\circ \frac{1}{L(u,P_u\theta)}L)$, one can use $\textstyle(\frac{1}{|P_u\theta|^2}O)^{k+1}$.
\end{remark}

\subsection{Inversion of the Weighted Cone Transform}
The following result presents a relation between the weighted divergent beam and cone transforms, which will be instrumental in the inversion of the latter one.
\begin{prop}\label{ConeFanInt}
Suppose that $f \in \mathcal{S}(\R^n)$, and $h(t)$ is a distribution on $\R$ regular near $t=\pm 1$. Then,
\begin{align}\label{ConeFanDual}
\big\langle  h(t),\frac{C^kf(u,\beta,t)}{\sqrt{1-t^2}} \big\rangle_\R =  \langle h(\sigma \cdot\beta), D_u^k f(\sigma)\rangle_{\S^{n-1}}.
\end{align}
(Notice that $C^kf(u,\beta,t)=0$ for $|t|>1$, and is smooth for $|t|<1$.)
\end{prop}
\begin{proof}
By the representation \eqref{Ct} of the weighted cone transform, we have
\begin{align*}
\big\langle  h(t),\frac{C^kf(u,\beta,t)}{\sqrt{1-t^2}} \big\rangle_\R
& =     \big\langle h(t),  \langle  \delta(\sigma \cdot\beta-t), D_u^k f(\sigma)\rangle_{\S^{n-1}} \big\rangle_\R\\
 &=  \big\langle (h \ast \delta)(\sigma \cdot\beta), D_u^k f(\sigma)\big\rangle_{\S^{n-1}}
 =  \langle h(\sigma \cdot\beta), D_u^k f(\sigma)\rangle_{\S^{n-1}}.
\end{align*}
\end{proof}
We note that when $h$ is a regular distribution near $t=\pm 1$, the identity \eqref{ConeFanDual} reads as
\begin{align}\label{ConeDivInt}
 \int_0^\pi C^kf(u,\beta,\psi)h(\cos \psi)d\psi =  \int_{\S^{n-1}} D_u^kf(\sigma)h(\sigma \cdot \beta) d\sigma.
\end{align}

\begin{definition}
Let $k<n-1$. We introduce the following distribution:
\begin{align}\label{E:distr_hnk}
h_{n,k}(t)
&:= \int_0^\infty e^{-its}s^{n-k-2}ds\\
&\;=i^{n-k-1}[(n-k-2)!t^{k-n+1}+(-1)^{n-k-1}i\pi\delta^{(n-k-2)}(t)],\nonumber
\end{align}
(see e.g. \cite[Chapter 2, p.172, eqn. (5)]{Gelfand}).
\end{definition}

\begin{prop}
The following identity holds:
\begin{align}\label{FDC}
\widehat{D_u^k f}(\xi)= \int_{\S^{n-1}} D_u^kf(\sigma)h_{n,k}(\sigma \cdot \xi) d\sigma.
\end{align}
\end{prop}
\begin{proof}
For each $\varphi \in \mathcal{S}({\R^n})$, since $D_u^kf$ is homogeneous of degree $-(k+1)$, we have
\begin{align*}
\langle \widehat{D_u^kf}, \varphi \rangle
& = \langle D_u^kf, \hat{\varphi} \rangle  = \int_{\R^n} D_u^k f(y) \hat{\varphi}(y) dy\\
& = \int_{\S^{n-1}} \int_0^\infty  D_u^k f(\sigma) \hat{\varphi}(s\sigma) s^{n-k-2}dsd\sigma \\
& = \int_{\S^{n-1}} \int_0^\infty  D_u^k f(\sigma) \int_{\R^n}  e^{-i s\sigma \cdot \xi} \varphi(\xi) d\xi s^{n-k-2}dsd\sigma\\
& = \int_{\R^n}  \int_{\S^{n-1}} D_u^k f(\sigma) h_{n,k}(\sigma \cdot \xi) d\sigma \varphi(\xi) d\xi,
\end{align*}
which implies \eqref{FDC}.
\end{proof}
Now, by combining \eqref{ConeDivInt} and \eqref{FDC}, and using the inversion formula for the weighted divergent beam transform \eqref{IDT}, we obtain an inversion formula for the weighted cone transform.
\begin{theorem}\label{T:ICT}
Let $f \in \mathcal{S}(\R^n)$, and $k<n-1$. Then,
\begin{align}\label{ICT}
f(x)= \frac{(-i)^{k+1}}{(2\pi)^n} \int_{\S^{n-1}} \int_0^\pi \big(\Delta_u^{(k+1)/2}C^kf(u,\beta,\psi)\big)\big|_{u=x} h_{n,k}(\cos \psi)d\psi d\beta.
\end{align}
\end{theorem}
\begin{remark}\label{R:defic_cone}
The same deficiency applies here that was mentioned in remark \ref{R:defic_beam}: the formula requires the cones to be available with all vertices $u$ throughout the space, which is unacceptable for many imaging applications (e.g. Compton ones). Fortunately, a similar remedy as for the divergent beam transform exists, which we address next.
\end{remark}

\begin{theorem}\label{T:C-sub}
Let $k \in \Z_+$. Suppose that $\M\subset\R^n$ satisfies the Tuy's condition with respect to a compact $V \subset \R^n$, and $f$ is a smooth function supported in $V$. Then, depending on the parity of $k$, the following inversion formulas hold:
\begin{enumerate}
\item if k is odd, then for any homogeneous linear elliptic differential operator $L(u,\partial_u)$ of order $k+1$ on $\M$
 \begin{align}\label{C-sub_odd}
 f(x)=\frac{1}{(2\pi)^n}\int_{\S^{n-1}} \int_0^\pi \frac{1}{L(u,P_u\theta)}L(u,\partial_u) C^kf(u,\beta,\psi)h_{n,k}(\cos \psi)d\psi d\theta,
 \end{align}
\item if k is even, then for any homogeneous linear elliptic differential operator $L(u,\partial_u)$ of order $k$ on $\M$
 \begin{align}\label{C-sub_even}
 f(x)=\frac{1}{(2\pi)^n}\int_{\S^{n-1}}\int_0^\pi \frac{1}{|P_u\theta|^2}\big(O\circ \frac{1}{L(u,P_u\theta)}L\big) C^kf(u,\beta,\psi) h_{n,k}(\cos \psi)d\psi d\theta,
 \end{align}
 where $O$ is given as in \eqref{E:OperatorO}, and $u \in \M$ is related to $x$ and $\theta$ as in the Tuy's condition.
 \end{enumerate}
\end{theorem}
\begin{proof}
The proof is just an immediate consequence of the equalities (\ref{FDC}), (\ref{ConeDivInt}), (\ref{D-sub}), and (\ref{D-sub-even}) (in that order).
\end{proof}

\section{Relations with the Radon Transform: Other Inversion Formulas}\label{S:InvRadon}
In this section, we present a relation between the Radon and weighted divergent beam and cone transforms, and from it we develop other analytical inversion formulas for the latter two in any dimension, when $k \in \mathbb{Z}_{+}$.

We recall first that the $n$-dimensional Radon transform $R$ maps a function $f$ on $\R^n $ into the set of its integrals over the hyperplanes of $\R^n $. Namely, if $\omega \in \S^{n-1}$ and $s \in \R$,
\begin{equation}
Rf(\omega,s) = R_\omega f(s) := \int_{x \cdot \omega=s} f(x) dx.
\end{equation}
In this notation, the Radon transform of $f$ is the integral of $f$ over the hyperplane perpendicular to $\omega$ at the signed distance $s$ from the origin.

We will use the following well known inversion formula for the Radon transform (see e.g. \cite{GGG,Helgason,Natt_old}). For any $f \in \mathcal{S}(\mathbb{R}^n)$,
\begin{align}\label{RadonInversion}
f(x)=\frac{(2\pi)^{1-n}}{2}
\begin{dcases}
(-1)^{(n-1)/2}\int_{\S^{n-1}} (Rf)^{(n-1)}(\omega, x \cdot \omega) d\omega, & \text{if $n$ is odd},\\
 (-1)^{(n-2)/2}\int_{\S^{n-1}} \mathcal{H}(Rf)^{(n-1)}(\omega, x \cdot \omega) d\omega, & \text{if $n$ is even},\\
\end{dcases}
\end{align}
where $\mathcal{H}$ is the Hilbert transform in $\R$ defined as the principal value integral

\begin{equation}\label{Hilbert}
\mathcal{H}g(t) = \frac{1}{\pi} p.v. \int_\R \frac{g(s)}{t-s} ds
\end{equation}
and
$$
(Rf)^{(n-1)}(\omega, s):=\frac{\partial^{n-1}}{\partial s^{n-1}}R(\omega, s).
$$
We now present a relation between the Radon and the weighted divergent beam and cone transforms for any dimension $n$ and any $k \in \mathbb{Z}_{+}$. Analogous relation for the usual divergent beam transform ($k=0$) is obtained in \cite{Hamaker} (see also \cite[Chapter 2]{NattWubb}).

Let $k \in \mathbb{Z}_{+}$, and $h$ be the function on $\R$ defined by
\begin{align}\label{h}
h(t):=
\begin{dcases}
\frac{1}{2(k-n+1)!}|t|^{k-n+1}\text{sgn}\;t, & \text{if $k > n-2$ and $k-n$ is odd},\\
\frac{1}{2(k-n+1)!}|t|^{k-n+1}, & \text{if $k > n-2$ and $k-n$ is even},\\
\delta^{(n-k-2)}(t), &  \text{if $k \leq n-2$}.
\end{dcases}
\end{align}

We note that $h$ is homogeneous of degree $k-n+1$, and for $k>n-2$, $h^{(k-n+2)}=\delta(t)$ (see e.g. \cite{Gelfand}).
\begin{prop}\label{ConeDivRadonProp}
Let $f \in \mathcal{S}(\R^n)$ and $h$ is given in \eqref{h}, then
\begin{align}\label{ConeDivRadon}
&\int_0^\pi C^kf(u,\beta,\psi)h(-\cos \psi)d\psi \\
 &=\int_{\S^{n-1}} D_u^kf(\sigma)h(-\sigma \cdot \beta) d\sigma = \int_{\mathbb{R}} Rf(\beta,s)h(u\cdot \beta-s) ds = (R_\beta f \ast h)(u\cdot \beta).\nonumber
\end{align}
\end{prop}

\begin{proof}
The first equality is already obtained in Proposition \ref{ConeFanInt}. By definition of the weighted divergent beam transform, we have
\begin{align*}
 \int_{\S^{n-1}} D_u^kf(\sigma)h(-\sigma \cdot \beta) d\sigma
&= \int_{\S^{n-1}} \int_0^\infty f(u+\rho \sigma) \rho^k d\rho h(-\sigma \cdot \beta) d\sigma \\
 &= \int_{\S^{n-1}} \int_0^\infty f(u+\rho \sigma) h(-\rho \sigma \cdot  \beta) \rho^{n-1} d\rho d\sigma,
 \end{align*}
due to the homogeneity of $h$. Letting $x=u+\rho\sigma$, we obtain
\vspace{-1ex}
 \begin{align*}
  \int_{\S^{n-1}}& \int_0^\infty  f(u+\rho \sigma) h(-\rho \sigma \cdot  \beta) \rho^{n-1} d\rho d\sigma \\
  &= \int_{\mathbb{R}^n} f(x) h((u-x) \cdot \beta) dx
  = \int_{\mathbb{R}^n} f(x) \left(\int_{\mathbb{R}} h(u \cdot \beta-s)  \delta(x \cdot \beta-s)ds\right) dx\\
 &= \int_{\mathbb{R}} \left(\int_{\mathbb{R}^n} f(x)\delta(x \cdot \beta-s) dx \right) h(u \cdot \beta-s) ds
= \int_{\mathbb{R}} Rf(\beta,s)h(u\cdot \beta-s) ds.
\end{align*}
\end{proof}
\begin{remark}\normalfont\indent
In dimension three, for $k=1$, the relation \eqref{ConeDivRadon} gives the following (geometrically obvious) formula:
$$C^1f(u,\beta,\frac{\pi}{2})=R(\beta, u\cdot \beta),$$
 which is used and in \cite{Basko}.
\end{remark}
\begin{remark}\normalfont\indent
We notice that Proposition \ref{ConeDivRadonProp} is valid for any choice of $h$ as long as it is homogeneous of degree $k-n+1$ and regular around $\pm1$. Indeed, $h(t)=t^{k-n+1}$ would work, too. In fact, the relation \eqref{ConeDivRadon} is proven for such $h$ and is used to derive an inversion formula for the cone transform for $k=0$ and $k=1$ in dimension three in \cite{Smith}, and for $k=0$ in dimension two in \cite{ADHKK,Hristova}. For the usual divergent beam transform ($k=0$) in dimension three, the applications of various functions $h$ can be found in \cite[Chapter 2, and references therein]{NattWubb}.

\end{remark}
Now, using the differentiation property of the convolution
$$\partial^\alpha(g \ast h) = \partial^\alpha g \ast h =  g \ast \partial^\alpha h$$
and the inversion formula \eqref{RadonInversion} for the Radon transform, we obtain the following formula, which can be used for inversion of both the weighted divergent beam and cone transforms:
\begin{theorem}\label{T:InviaRadon}
Suppose that for any $u \in \M$ and $\beta \in \S^{n-1}$, $s = u \cdot \beta$ and
\begin{align}\label{defG}
G(s,\beta):= (R_\beta f \ast h)(s) &= \int_0^\pi C^kf(u,\beta,\psi)h(-\cos \psi)d\psi \\
&=\int_{\S^{n-1}} D_u^kf(\sigma)h(-\sigma \cdot \beta) d\sigma.\nonumber
\end{align}
Then, for any $f \in \mathcal{S}(\mathbb{R}^n)$,
\begin{align}\label{ConeInversion}
f(x)= \frac{1}{2}(2\pi)^{1-n}
\begin{dcases}
(-1)^{(n-1)/2} \int_{\S^{n-1}} G^{(k+1)}(x \cdot \beta,\beta) \;d\beta, & \text{if $n$ is odd},\\
(-1)^{(n-2)/2} \int_{\S^{n-1}} \mathcal{H} G^{(k+1)}(x \cdot \beta,\beta)\; d\beta, & \text{if $n$ is even},
\end{dcases}
\end{align}
where $G^{(k+1)}$ is the $(k+1)$-st derivative of $G$ with respect to $s$, $h$ is given in \eqref{h}, and $\mathcal{H}$ is the Hilbert transform \eqref{Hilbert} with respect to $s$.
\end{theorem}

\begin{proof}
Using $h$ given in \eqref{h} for $k \geq n-2$, we have
\begin{align*}
(R_\beta f \ast h)^{(k+1)}= (R_\beta f)^{(n-1)} \ast h^{(k-n+2)}= (R_\beta f)^{(n-1)} \ast \delta = (R_\beta f)^{(n-1)},
\end{align*}
and for $k < n-2$, we get
\begin{align*}
(R_\beta f \ast h)^{(k+1)}= (R_\beta f \ast \delta^{(n-k-2)})^{(k+1)} = (R_\beta f)^{(n-1)}.
\end{align*}
Hence, the application of Radon transform inversion \eqref{RadonInversion} gives the result.
\end{proof}

\begin{remark}\normalfont\indent
The reconstruction formula \eqref{ConeInversion} is independent of the geometry of the manifold $\M$ that $u$ belongs to. Indeed, for the weighted cone transform, it is sufficient that for any $s \in \R$ and $\beta \in \S^{n-1}$, there is a vertex $u = s \beta+y$, where  $y \bot \beta$, and the weighted cone data is available at all angles $\psi$ for these $u$ and $\beta$. In other words, the requirement for the reconstruction is that any hyperplane intersecting the domain of reconstruction contains the vertex of a cone with the axis normal to the plane and all opening angles.

For the weighted divergent beam transform, the corresponding condition is that for any $s \in \R$ and $\beta \in \S^{n-1}$, there is a source $u = s \beta+y$, where  $y \bot \beta$, and the weighted divergent beam data is available at all directions $\sigma$ for this source $u$.
\end{remark}
\section{Numerical Implementation of Theorem \ref{T:InviaRadon}}\label{S:numerics}
In this section, we present the results of numerical implementation of Theorem \ref{T:InviaRadon}. In dimension two, the weighted cone transform and divergent beam transforms are similar, so we only provide examples for the weighted cone transform for $k=1$ using two different vertex geometries. We then give the reconstruction results for the weighted cone transform in dimension three for $k=0$ and $k=2$ using a spherical vertex geometry. Examples showing the reaction of the algorithms to Gaussian white noise in the data are also provided. We also present an example of numerical inversion of the weighted divergent beam transform in dimension three for the cases $k=1$ and $k=2$ using a spherical source geometry. 

All phantoms are placed off-center of the vertex curve/surface, to avoid unintended use of rotational invariance. Care was taken to avoid other possibilities of committing an inverse crime, by making the forward and inverse algorithms as unrelated as possible.
\subsection{2D Image Reconstruction from Cone Data for k=1}
In dimension two, for $k = 0$, the relation \eqref{ConeDivRadon} gives $C^0(u,\beta,\pi/2)=Rf(\beta, u \cdot \beta)$, which is geometrically obvious (also see \cite{Basko}). Thus, we focus here on the case $k=1$ only. Since a cone in 2D is represented by two rays with a common vertex, the weighted cone transform for $k=1$ is given by
\begin{align*}
C^1&f(u,\beta(\phi),\psi) \\
&= \int_0^\infty [f(u+r(\cos(\phi-\psi),\sin(\phi-\psi))+f(u+r(\cos(\phi+\psi),\sin(\phi+\psi))] r dr,
\end{align*}
where $\beta(\phi)=(\cos(\phi),\sin(\phi))$. The inversion formula \eqref{ConeInversion} now reads as
\begin{align}\label{ConeInversion2D}
 f(x) = \frac{-1}{8\pi} \int_{\S^1} \Big(\mathcal{H} \frac{\partial^2}{\partial s^2} \int_0^\pi C^1 f(s\beta+y,\beta,\psi)\text{sgn}(\cos \psi)d\psi\Big) \Big|_{s=x \cdot \beta} d\beta.
\end{align}

For the numerical implementation of \eqref{ConeInversion2D}, we considered the phantom
$$f = \raisebox{2pt}{$\chi$}_{D_1} - 0.5 \raisebox{2pt}{$\chi$}_{D_2},$$
 where $D_1$ and $D_2$ are the concentric disks centered at $(0,0.4)$ with radii 0.25 and 0.5, respectively. Here, $\raisebox{2pt}{$\chi$}_{D_i}$ denotes the characteristic function of each disk (see Fig. \ref{fig:phantom2D}). The cone projection data of the phantom $f$ is simulated numerically using 256 counts for vertices $u$, 400 counts for central axis directions $\beta$ and 90 counts for opening angles $\psi$.
\begin{figure}[H]
\begin{center}
                 \begin{subfigure}{0.45\textwidth}
                \includegraphics[width=\textwidth]{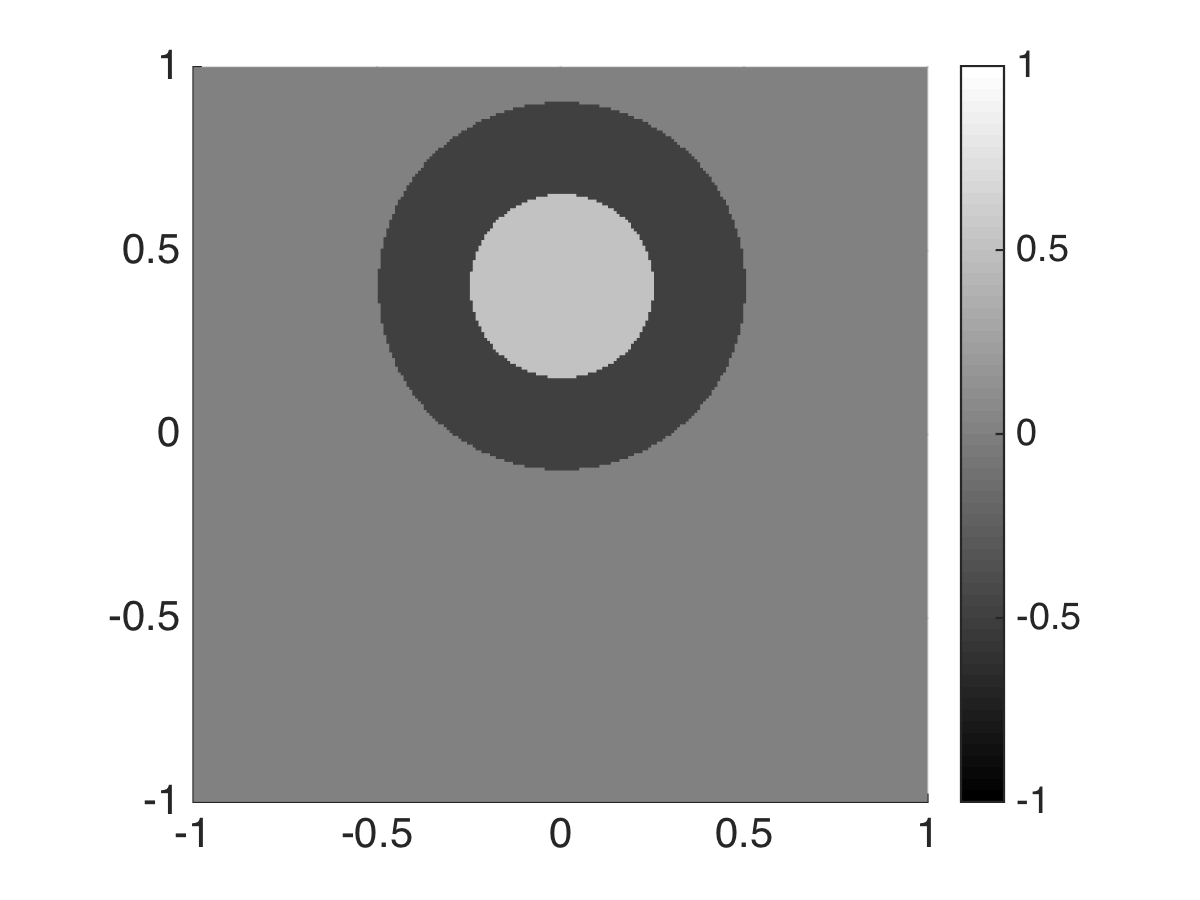}
        \end{subfigure}
        \hspace{-1em}
        \begin{subfigure}{0.5\textwidth}
                \includegraphics[width=\textwidth]{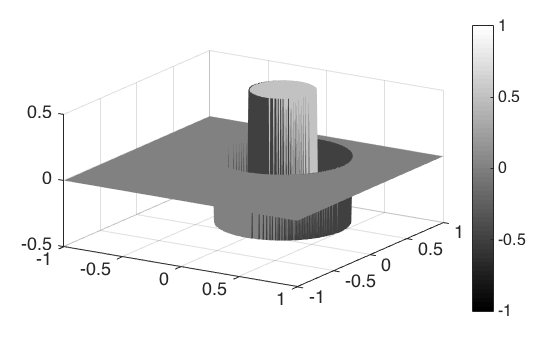}
        \end{subfigure}
        \caption{The density plot (left) and surface plot (right) of the phantom $f$ that consists of two concentric disks centered at $(0,0.4)$ with radii 0.25 and 0.5, and densities 1 and -0.5 units, respectively.}
        \label{fig:phantom2D}
\end{center}
\end{figure}
\begin{figure}[H]
\begin{center}
        \begin{subfigure}{0.45\textwidth}
                \includegraphics[width=\textwidth]{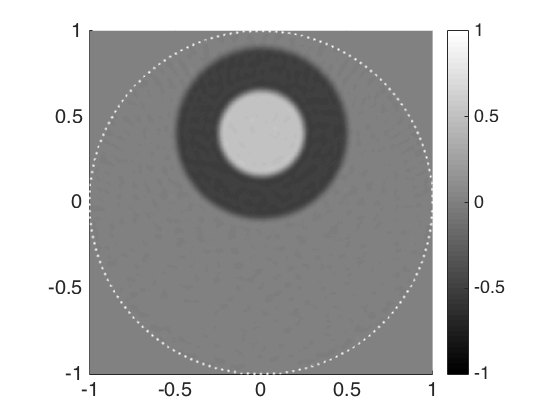}
        \end{subfigure}
        \begin{subfigure}{0.45\textwidth}
                \includegraphics[width=\textwidth]{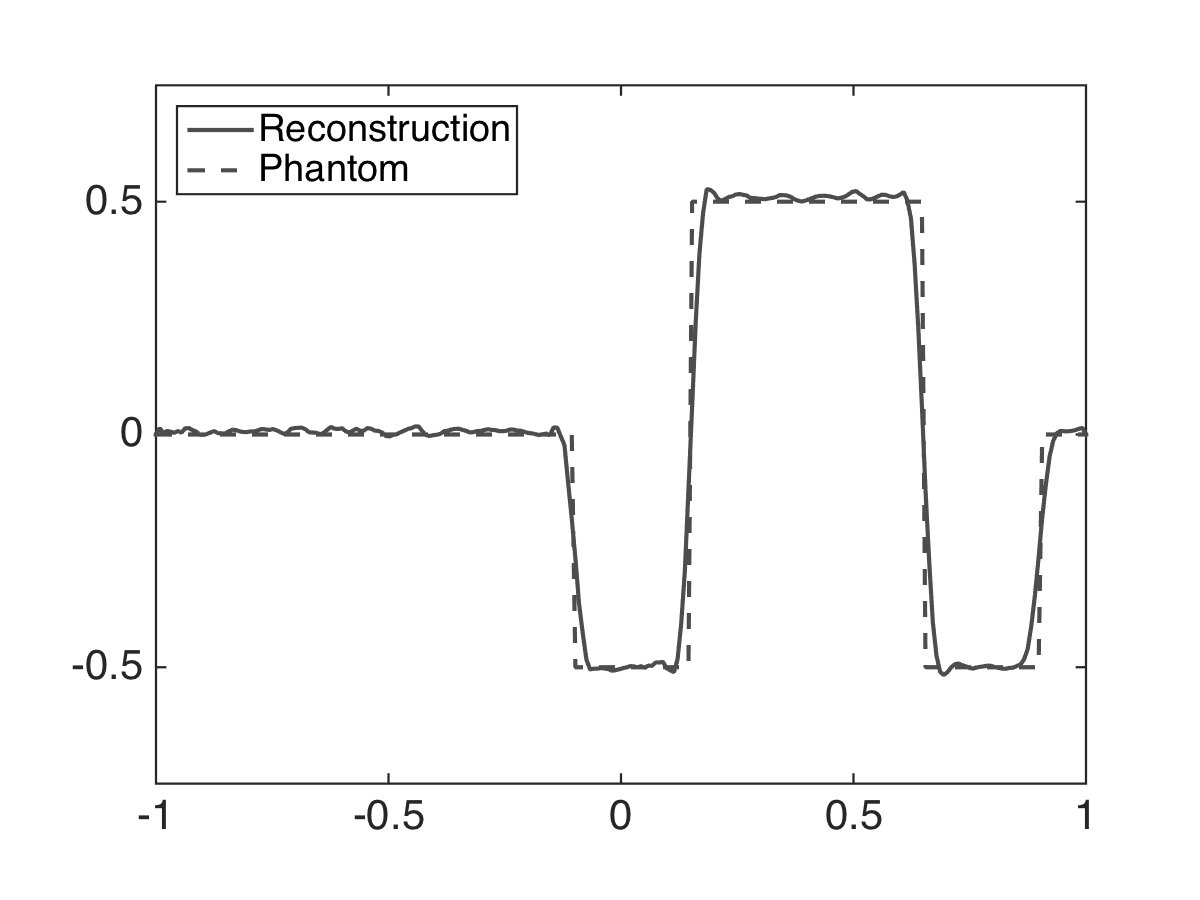}
        \end{subfigure}
  \caption{The density plot of $256 \times 256$ image reconstructed from the simulated cone data using 256 counts for vertices $u$ (represented by white dots on the unit circle), 400 counts for directions $\beta$ and 90 counts for opening angles $\psi$ (left), and the comparison of $y$-axis profiles of the phantom and the reconstruction (right).}
        \label{fig:2Dcircle}
\end{center}
\end{figure}
\begin{figure}[H]
\begin{center}
        \begin{subfigure}{0.45\textwidth}
                \includegraphics[width=\textwidth]{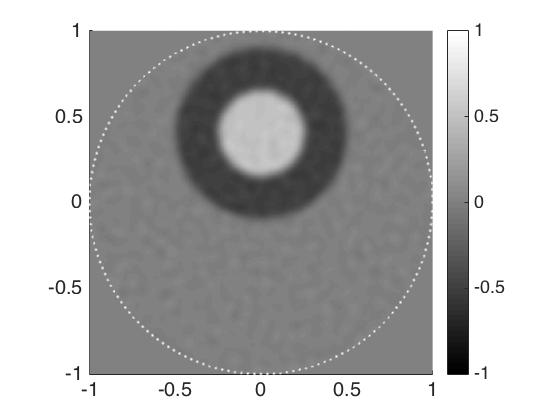}
        \end{subfigure}
        \begin{subfigure}{0.45\textwidth}
                \includegraphics[width=\textwidth]{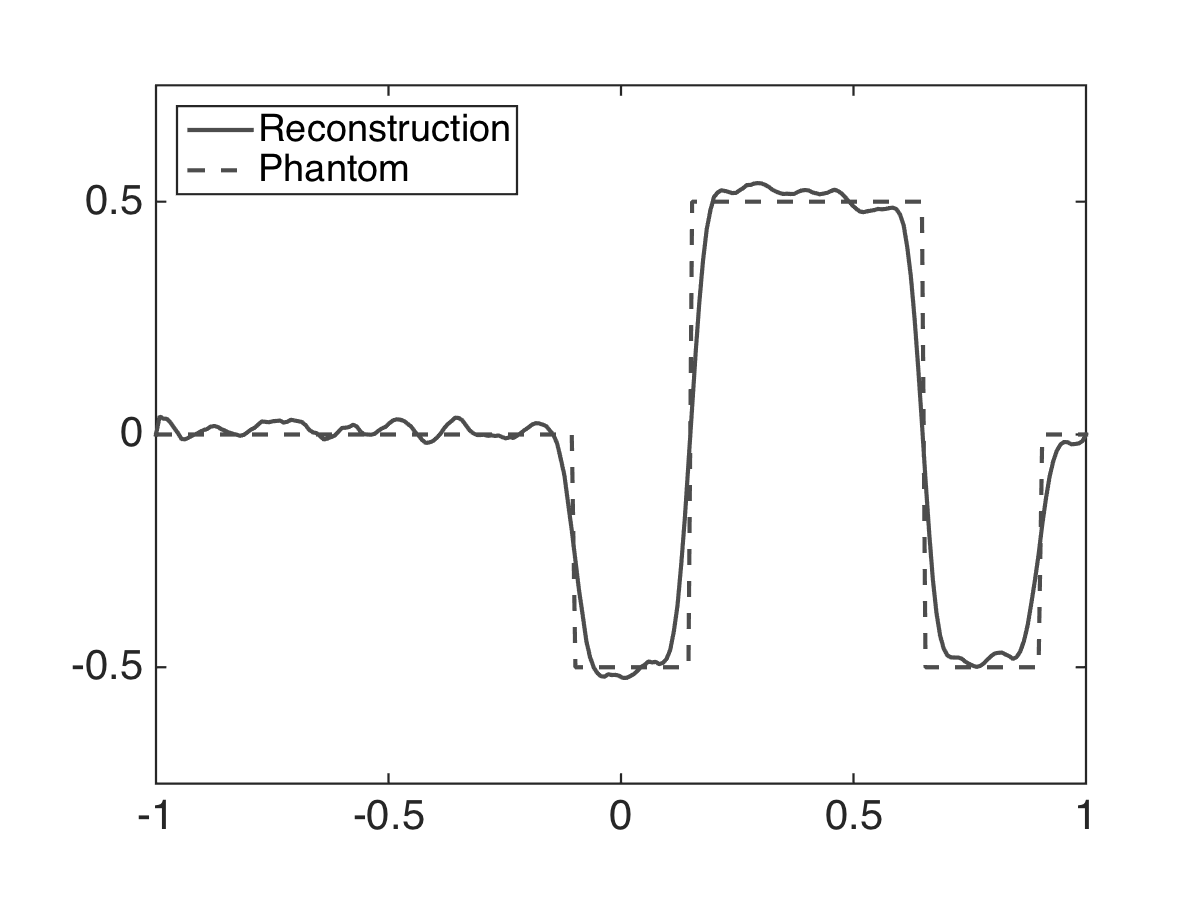}
        \end{subfigure}
  \caption{The density plot of $256 \times 256$ image reconstructed from cone data contaminated with $5\%$ Gaussian noise (left), and the comparison of $y$-axis profiles of the phantom and the reconstruction (right). The dimensions of the cone data are taken as in Fig. \ref{fig:2Dcircle}.}
        \label{fig:2DcircleNoisy}
\end{center}
\end{figure}
As our inversion formula is valid for arbitrary geometry of vertices, we considered both a circular and a square geometry of vertices of the cones.

 Figure \ref{fig:2Dcircle} shows the results of reconstruction from cone projections where the vertices cover the unit circle. The density plot and the $y$-axis profile of the reconstruction are provided in (a) and (b), respectively. The results with a $5\%$ Gaussian white noise added to the cone data is shown in Figure \ref{fig:2DcircleNoisy}.

In Figure \ref{fig:2Dsquare}, we provide the results of reconstruction from cone projections where the vertices are placed along the sides of a square with sides of length two. The density plot and the $y$-axis profile of the reconstruction are provided in Figure \ref{fig:2Dsquare} (a) and (b), respectively. Figure \ref{fig:2DsquareNoisy} shows the results with a $5\%$ Gaussian white noise added to the cone data.
\begin{figure}[H]
\begin{center}
        \begin{subfigure}{0.45\textwidth}
                \includegraphics[width=\textwidth]{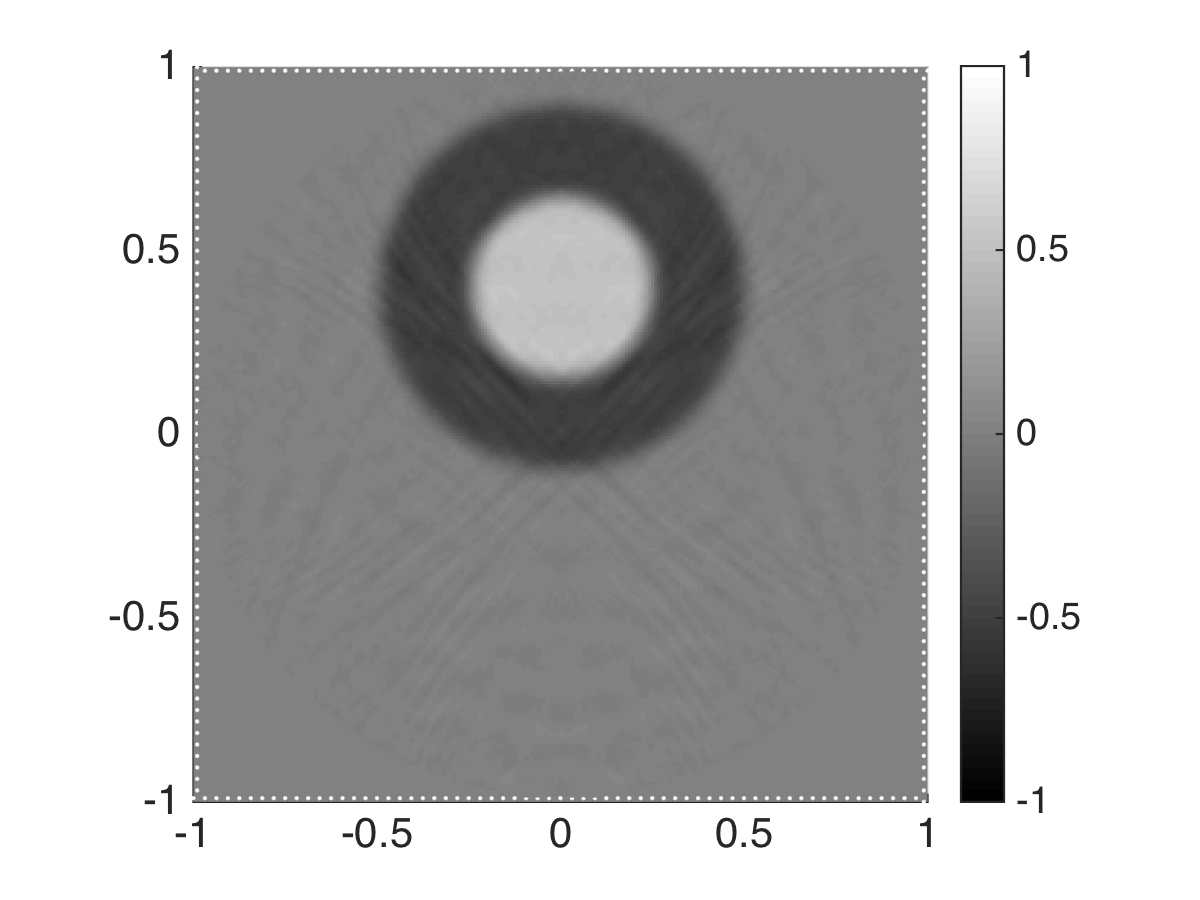}
        \end{subfigure}
        \begin{subfigure}{0.45\textwidth}
                \includegraphics[width=\textwidth]{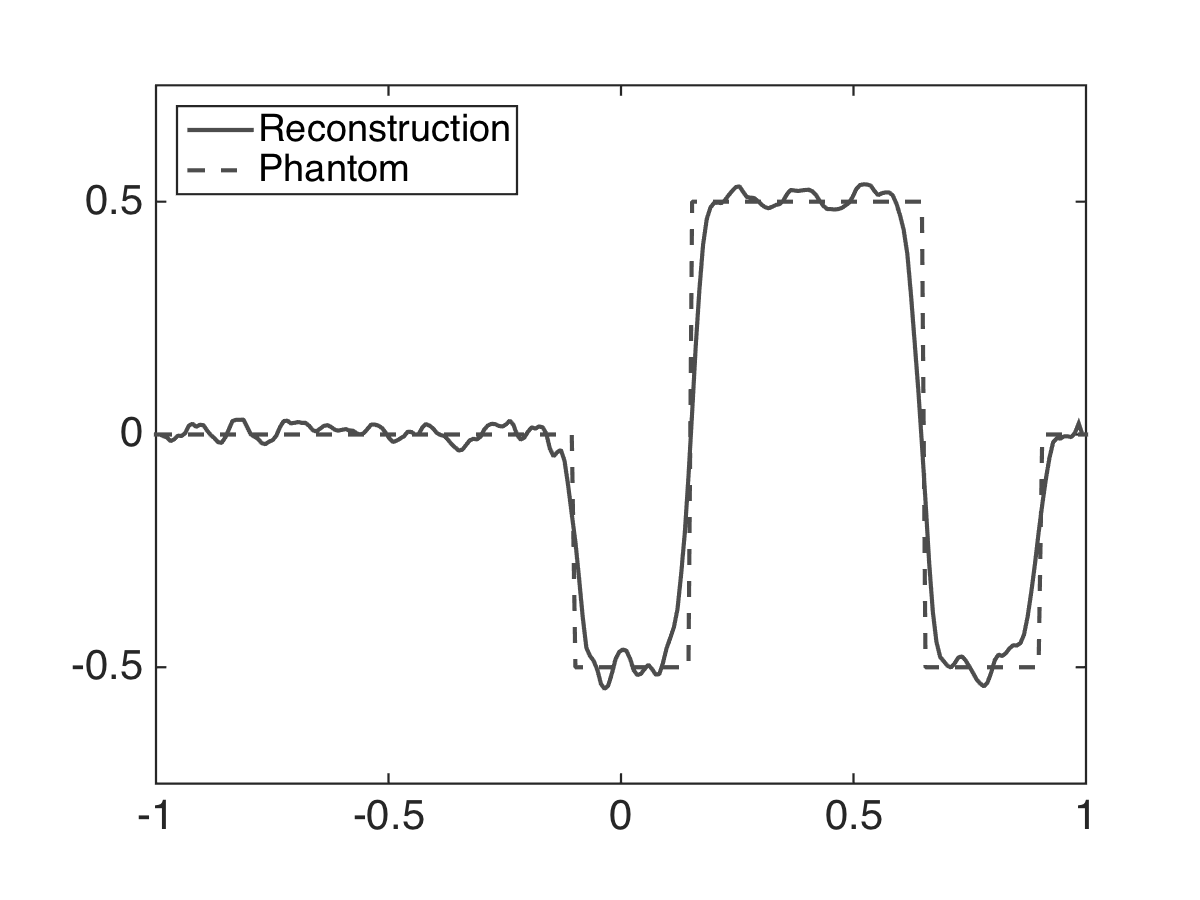}
        \end{subfigure}
  \caption{The density plot of $256 \times 256$ image reconstructed from the simulated cone data using 256 counts for vertices $u$ (represented by white dots around the square), 400 counts for directions $\beta$ and 90 counts for opening angles $\psi$ (left), and the comparison of $y$-axis profiles of the phantom and the reconstruction (right).}
        \label{fig:2Dsquare}
\end{center}
\end{figure}
\begin{figure}[H]
\begin{center}
        \begin{subfigure}{0.45\textwidth}
                \includegraphics[width=\textwidth]{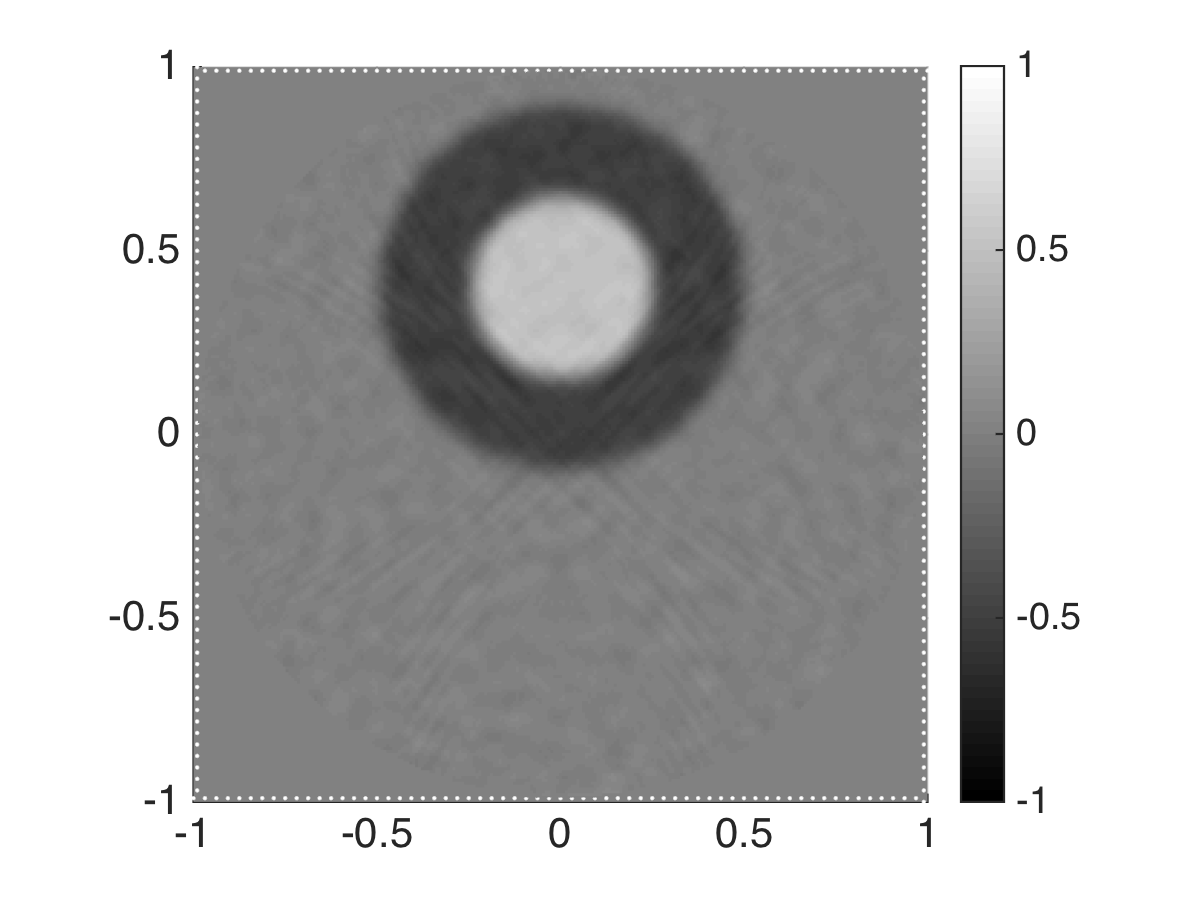}
        \end{subfigure}
        \begin{subfigure}{0.45\textwidth}
                \includegraphics[width=\textwidth]{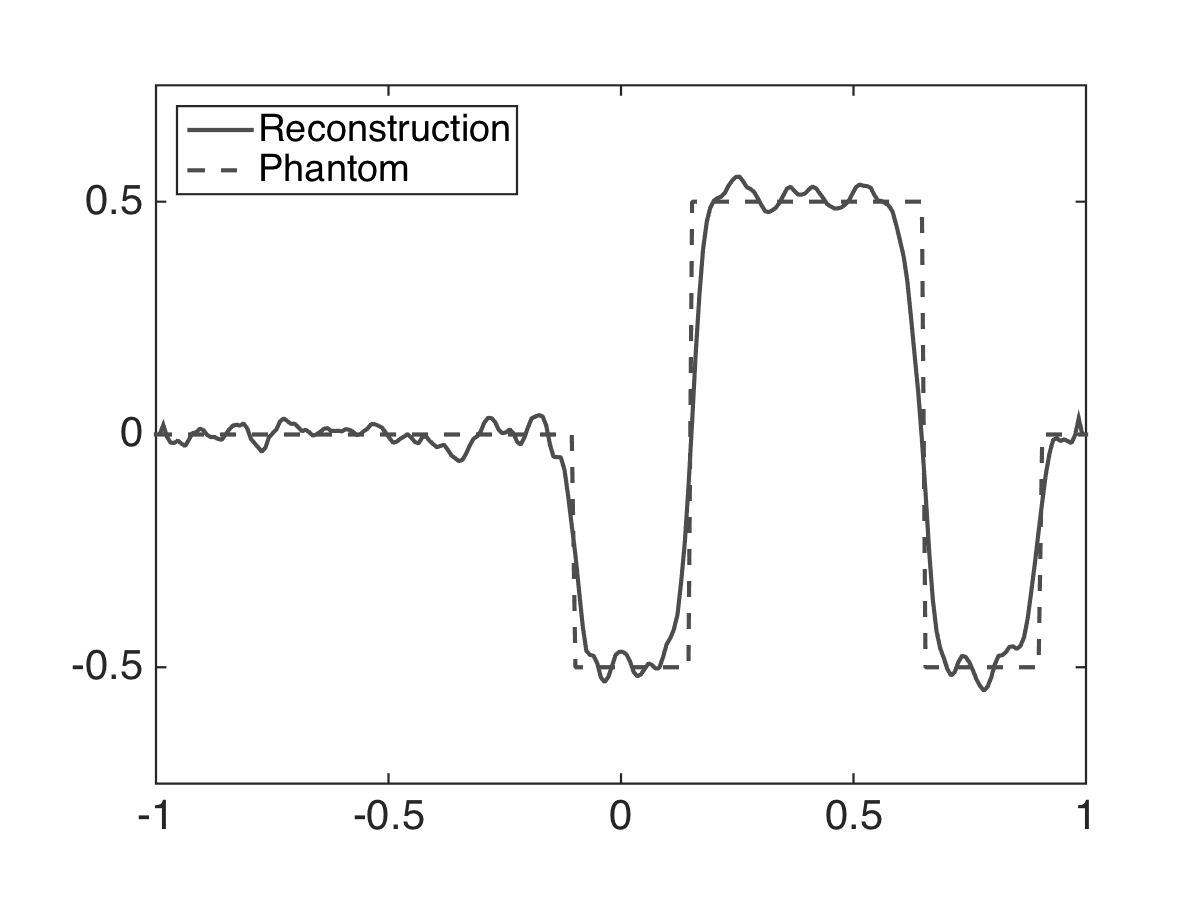}
        \end{subfigure}
  \caption{The density plot of $256 \times 256$ image reconstructed from cone data contaminated with $5\%$ Gaussian noise (left), and the comparison of $y$-axis profiles of the phantom and the reconstruction (right). The dimensions of the cone data are taken as in Fig. \ref{fig:2Dsquare}.}
        \label{fig:2DsquareNoisy}
\end{center}
\end{figure}
In the case of the square geometry (but not in the circular one), some corner-related effects appear along the diagonals, as shown in Figure \ref{fig:diagprofiles}. They can be eliminated by using a finer discretization in $\beta$.
\begin{figure}[H]
\begin{center}
                 \begin{subfigure}{0.45\textwidth}
                \includegraphics[width=\textwidth]{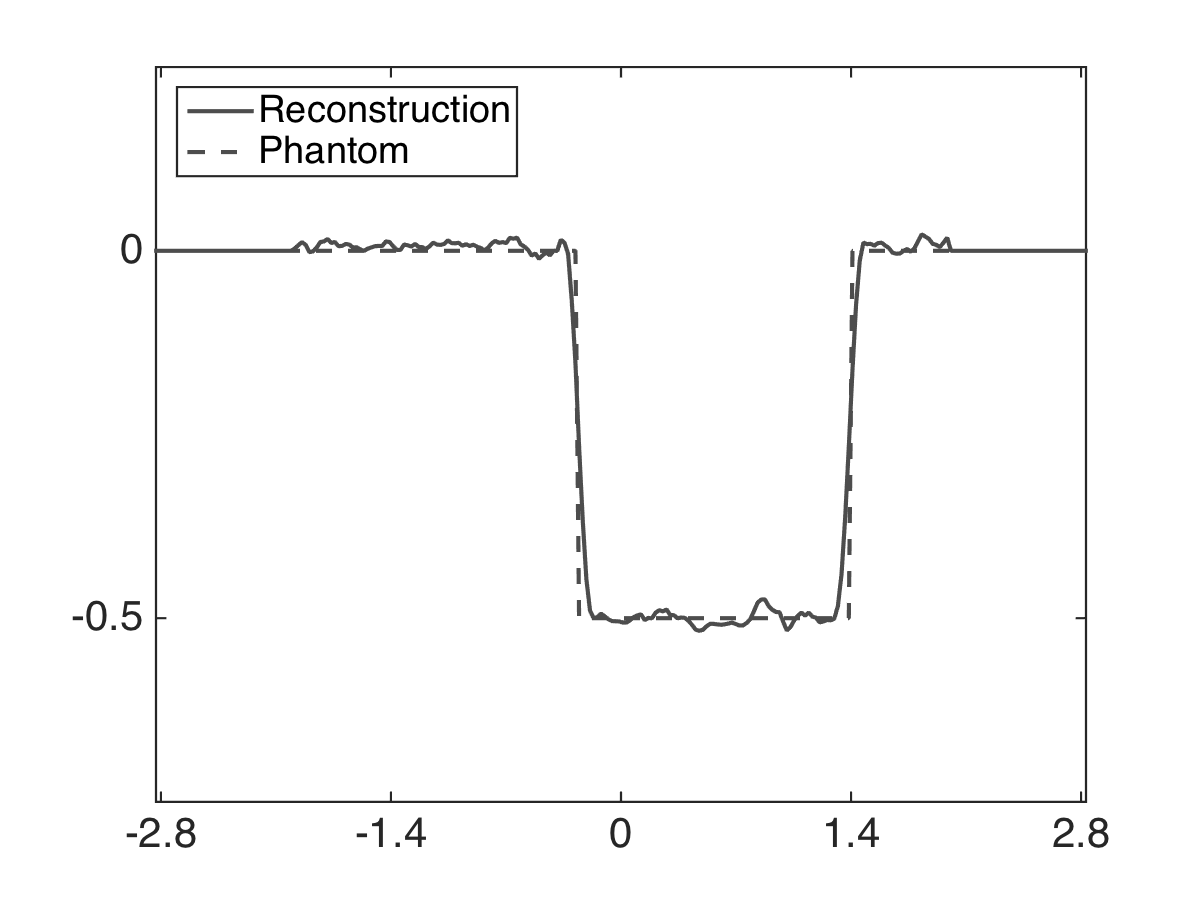}
        \end{subfigure}
        \begin{subfigure}{0.45\textwidth}
                \includegraphics[width=\textwidth]{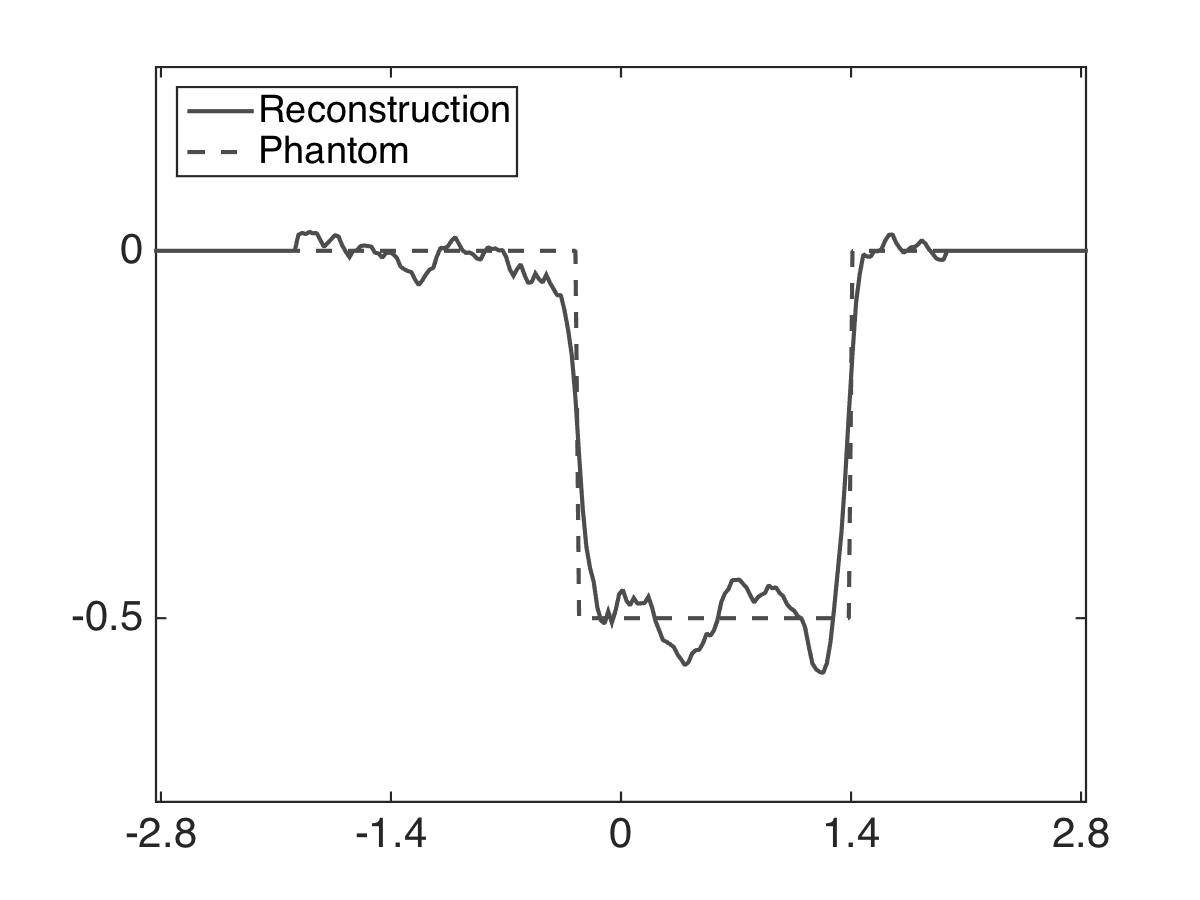}
        \end{subfigure}
        \caption{Comparison of the profiles of the reconstruction along the diagonal of the square region for the circular (left) and square (right) locations of the vertices (detectors).}
        \label{fig:diagprofiles}
\end{center}
\end{figure}
\subsection{3D Image Reconstruction from Weighted Cone Data for k=0 and k=2}
In dimension three, the relation \eqref{ConeDivRadon} is geometrically obvious for the case k=1 (see remark 4.2), and is numerically implemented in \cite{Basko} using spherical harmonic expansions. Here, we provide examples of reconstruction from the weighted cone data for $k=0$ and $k=2$. Theorem \ref{T:InviaRadon} gives the following inversion formula for $k=0$:
\begin{align}\label{ConeInversion3Dk0}
 f(x) &= \frac{1}{8\pi^2} \int_{\S^2} \Big(\frac{\partial}{\partial s} \int_0^\pi C^0 f(s\beta+y,\beta,\psi)\delta'(\cos \psi)d\psi \Big)\Big|_{s=x \cdot \beta} d\beta \nonumber\\
 &= \frac{-1}{8\pi^2} \int_{\S^2} \Big(\frac{\partial}{\partial s} \Big( \frac{\partial}{\partial t} C^0 f(s\beta+y,\beta,t) \Big)\Big|_{t=0}\Big)\Big|_{s=x \cdot \beta} d\beta,
\end{align}
and for $k=2$, we have
\begin{align}\label{ConeInversion3Dk2}
 f(x) = \frac{-1}{16\pi^2} \int_{\S^2} \Big(\frac{\partial^3}{\partial s^3} \int_0^\pi C^2 f(s\beta+y,\beta,\psi)\text{sgn}(\cos \psi)d\psi \Big)\Big|_{s=x \cdot \beta} d\beta.
\end{align}

In our examples, the vertices of the cones cover the unit sphere $\mathbb{S}^2$ in $\R^3$ and the phantom is the characteristic function of the 3D ball of radius 0.5 units located strictly inside and off-center of this sphere.

The forward simulations of weighted cone projections were done numerically using 1800 counts for vertices $u$ on the unit sphere, 1800 counts for unit vectors for the cone axis directions $\beta$ and 200 counts for opening angles $\psi$. For the discretization of the sphere, we used a uniform mesh for both the azimuthal and the polar angles.

Figure \ref{fig:3Dk0} shows the three cross sections of the spherical phantom and of its reconstructions from the cone data obtained via \eqref{ConeInversion3Dk0}. The comparison of the phantom and the reconstruction given in Figure \ref{fig:3Dk0} in terms of their coordinate axis profiles is provided in Figure \ref{fig:3Dk0profile}.
\begin{figure}[H]
\begin{center}
                 \begin{subfigure}{0.5\textwidth}
                \includegraphics[width=\textwidth]{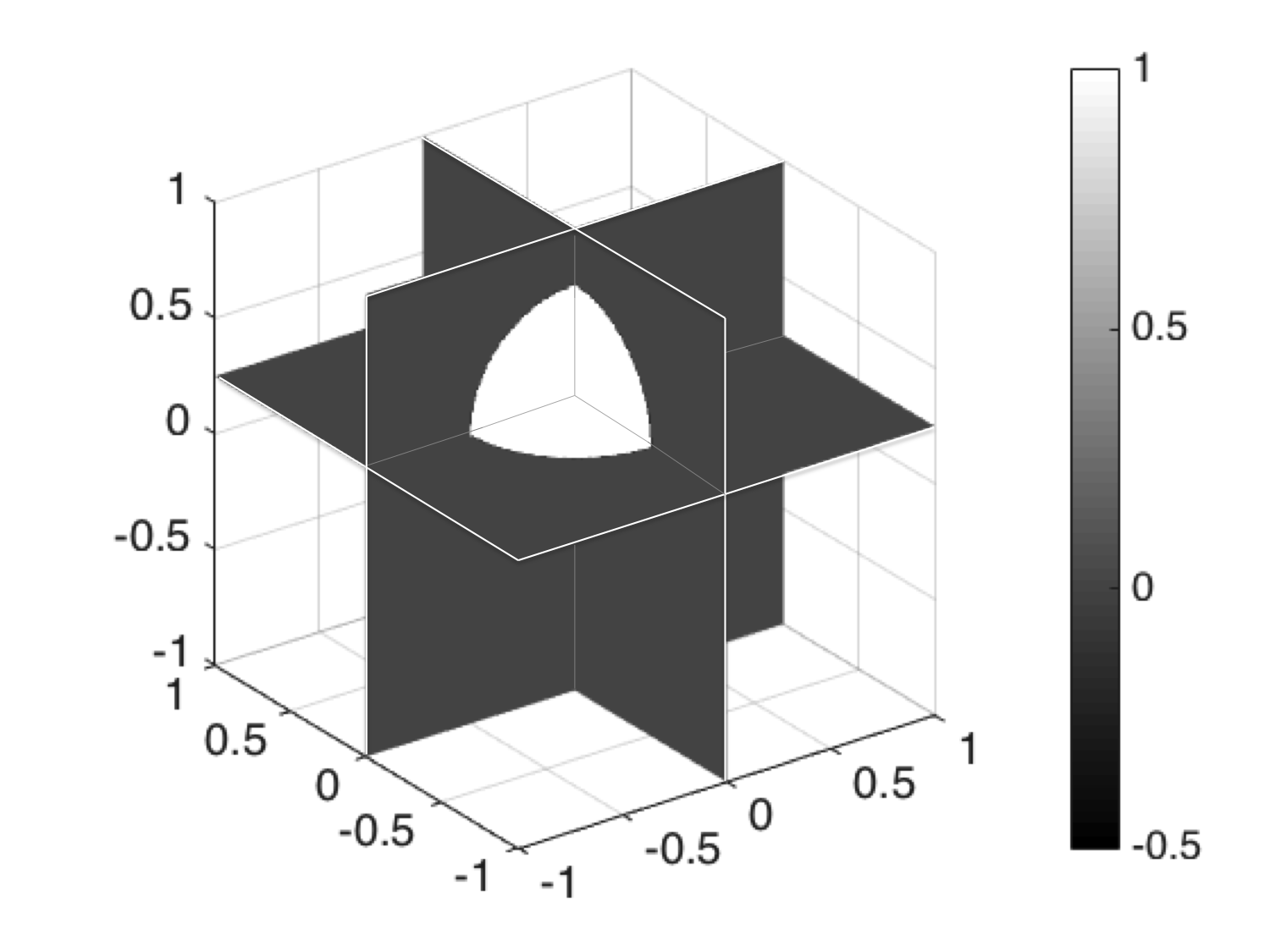}
        \end{subfigure}
\hspace{-1em}
        \begin{subfigure}{0.5\textwidth}
                \includegraphics[width=\textwidth]{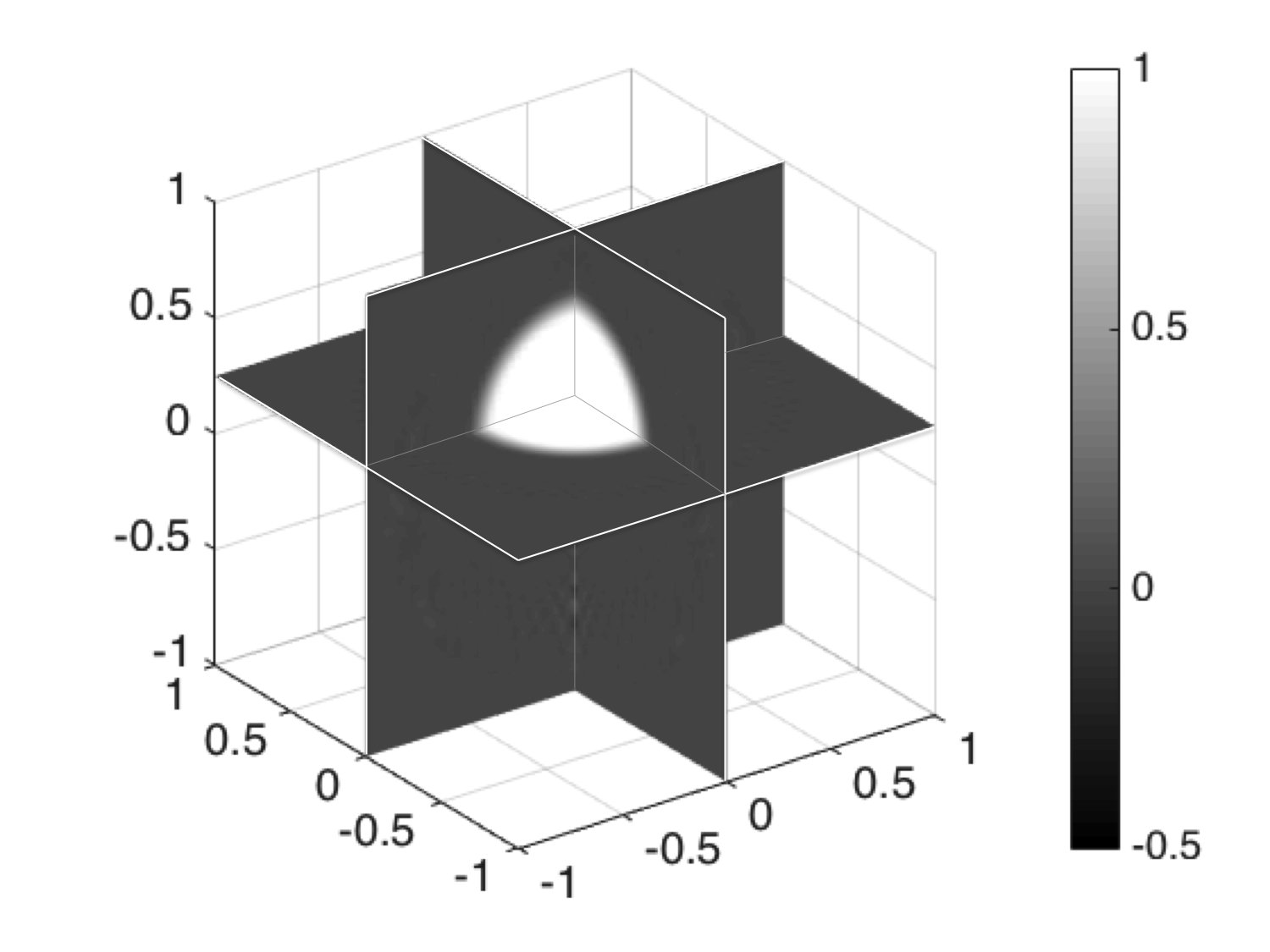}
        \end{subfigure}
       \caption{The 3D ball phantom with radius 0.5, center (0,0,0.25) and unit density (left), and $90 \times 90$ image reconstructed via \eqref{ConeInversion3Dk0} from weighted cone data simulated using 1800 counts for vertices $u$ on the unit sphere, 1800 counts for directions $\beta$ and 200 counts for opening angles $\psi$ (right). The cross sections by the planes $x=0, y=0$ and $z=0.25$ are shown.}
        \label{fig:3Dk0}
\end{center}
\end{figure}

\begin{figure}[H]
\begin{center}
             \begin{subfigure}{0.32\textwidth}
                \includegraphics[width=\textwidth]{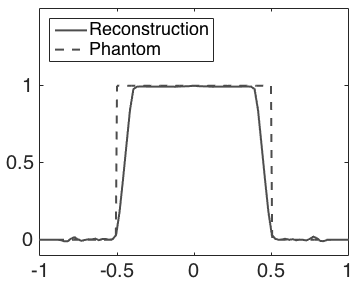}
        \end{subfigure}
  \begin{subfigure}{0.32\textwidth}
                \includegraphics[width=\textwidth]{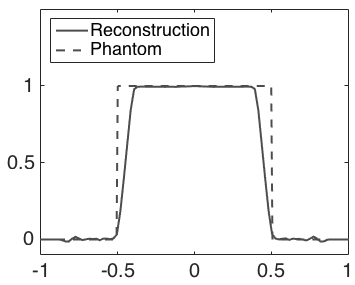}
        \end{subfigure}
  \begin{subfigure}{0.32\textwidth}
                \includegraphics[width=\textwidth]{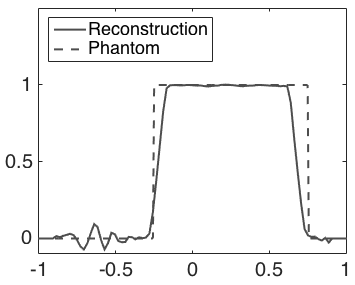}
        \end{subfigure}
        \caption{Comparison of the $x$-axis (left), $y$-axis (center) and $z$-axis (right) profiles of the reconstruction and the phantom given in Fig. \ref{fig:3Dk0}.}
       \label{fig:3Dk0profile}
\end{center}
\end{figure}

Figure \ref{fig:3Dk2} shows the three cross sections of the spherical phantom and of its reconstructions from the cone data obtained via \eqref{ConeInversion3Dk2}. The comparison of the phantom and the reconstruction given in Figure \ref{fig:3Dk2} in terms of their coordinate axis profiles is provided in Figure \ref{fig:3Dk2profile}.
\begin{figure}[H]
\begin{center}
                 \begin{subfigure}{0.5\textwidth}
                \includegraphics[width=\textwidth]{Phantom3D.png}
        \end{subfigure}
\hspace{-1em}
        \begin{subfigure}{0.5\textwidth}
                \includegraphics[width=\textwidth]{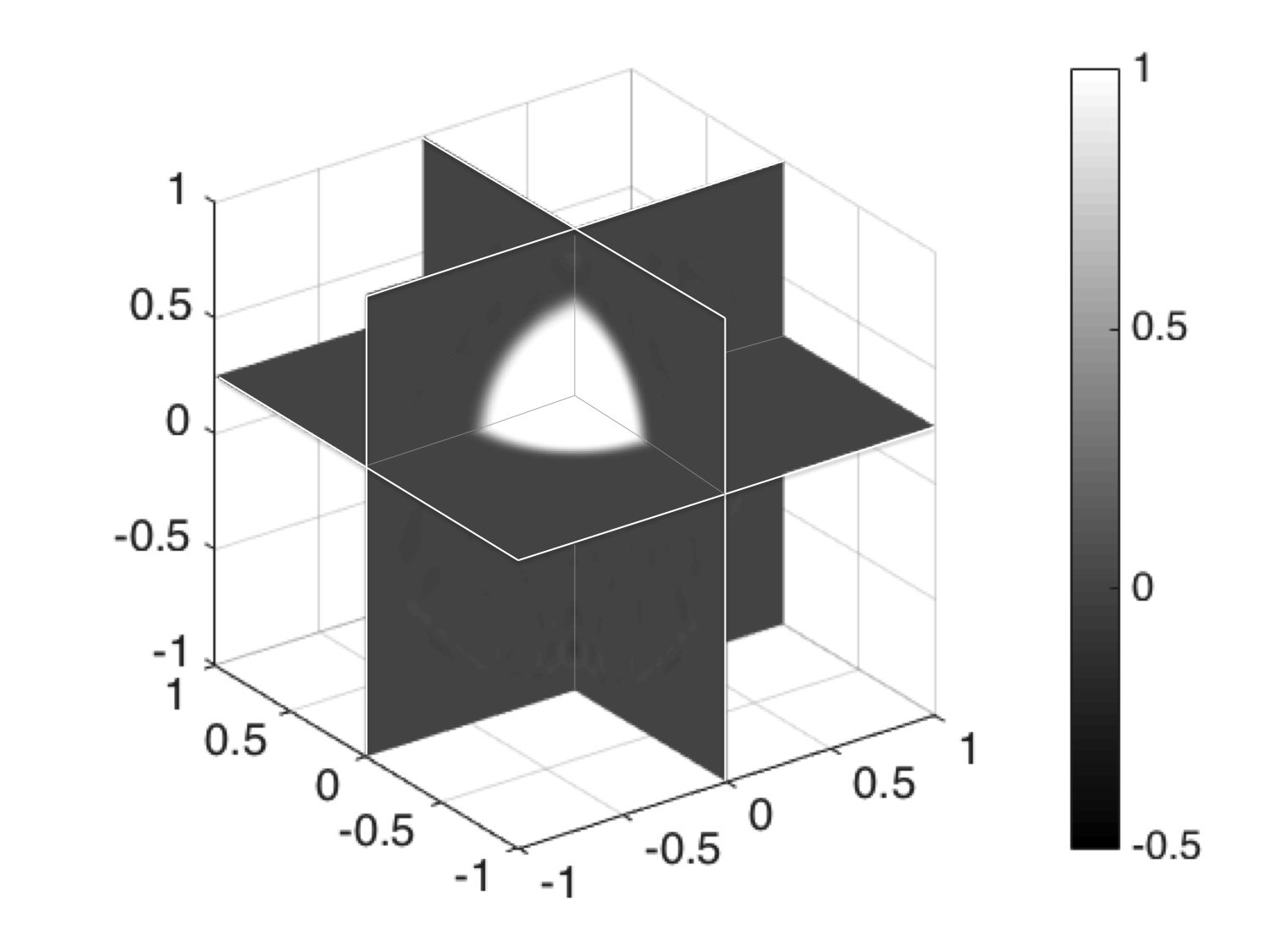}
        \end{subfigure}
       \caption{The 3D ball phantom with radius 0.5, center (0,0,0.25) and unit density (left), and $90 \times 90$ image reconstructed via \eqref{ConeInversion3Dk2} from weighted cone data simulated using 1800 counts for vertices $u$ on the unit sphere, 1800 counts for directions $\beta$ and 200 counts for opening angles $\psi$ (right). The cross sections by the planes $x=0, y=0$ and $z=0.25$ are shown.}
        \label{fig:3Dk2}
\end{center}
\end{figure}

\begin{figure}[H]
\begin{center}
             \begin{subfigure}{0.32\textwidth}
                \includegraphics[width=\textwidth]{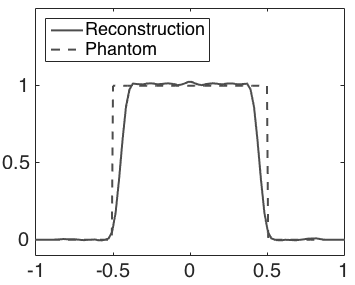}
        \end{subfigure}
  \begin{subfigure}{0.32\textwidth}
                \includegraphics[width=\textwidth]{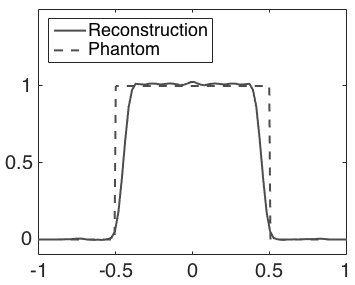}
        \end{subfigure}
  \begin{subfigure}{0.32\textwidth}
                \includegraphics[width=\textwidth]{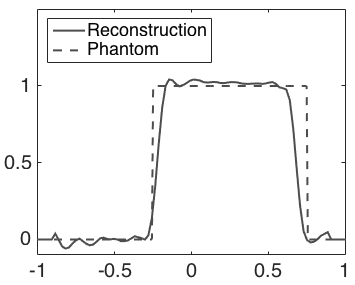}
        \end{subfigure}
        \caption{Comparison of the $x$-axis (left), $y$-axis (center) and $z$-axis (right) profiles of the reconstruction and the phantom given in Fig. \ref{fig:3Dk2}.}
       \label{fig:3Dk2profile}
\end{center}
\end{figure}
Figures \ref{fig:3Dnoisy} and \ref{fig:3Dnoisyprofile} show the results of reconstruction from weighted cone data for $k=2$ contaminated with $5\%$ Gaussian white noise.
\begin{figure}[H]
\begin{center}
                 \begin{subfigure}{0.5\textwidth}
                \includegraphics[width=\textwidth]{Phantom3D.png}
        \end{subfigure}
\hspace{-1em}
        \begin{subfigure}{0.5\textwidth}
                \includegraphics[width=\textwidth]{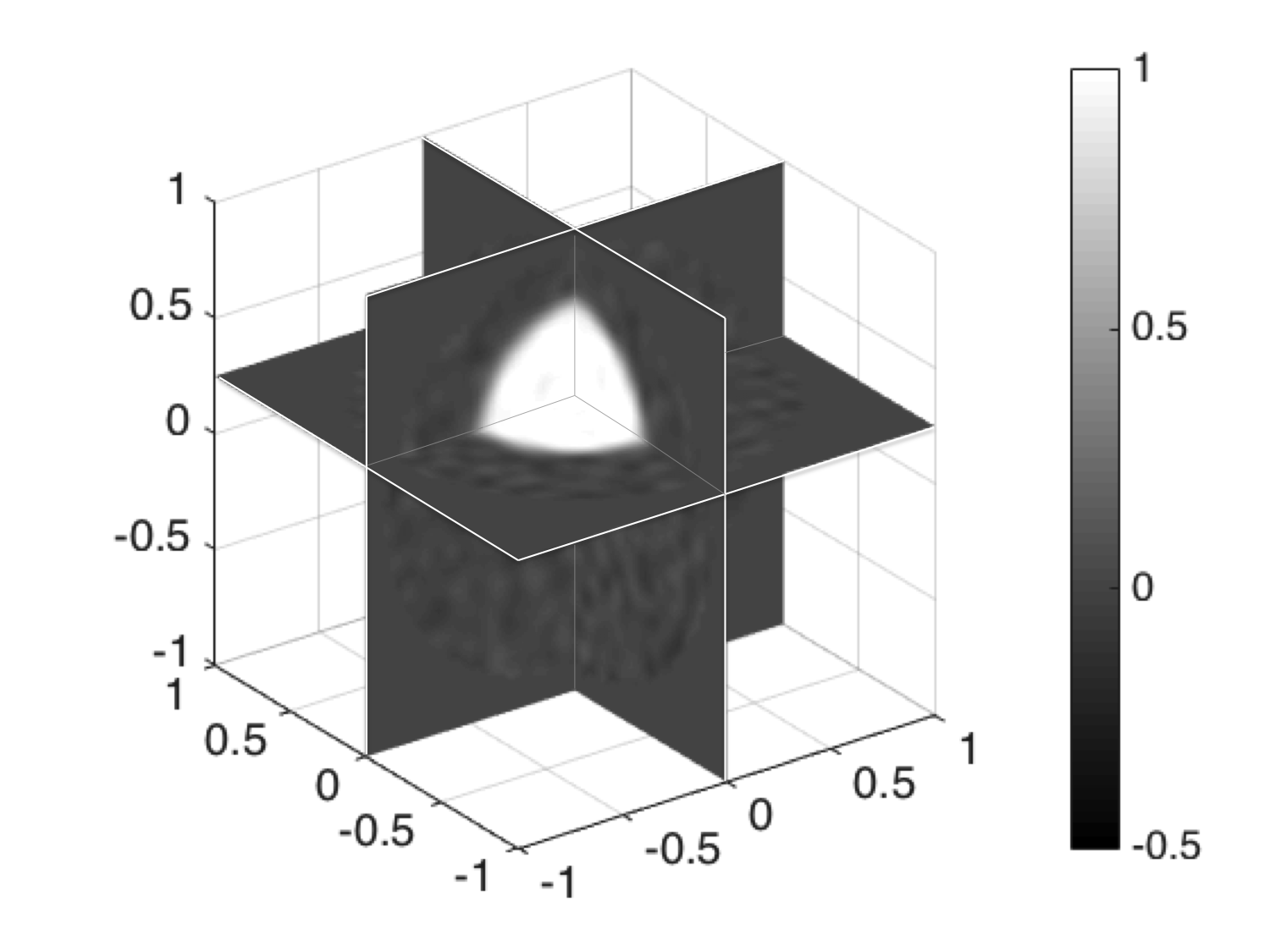}
        \end{subfigure}
       \caption{The 3D ball phantom with radius 0.5, center (0,0,0.25) and unit density (left), and $90 \times 90$ image reconstructed via \eqref{ConeInversion3Dk2} from weighted cone data contaminated with $5\%$ Gaussian white noise (right). The dimensions of the cone projections are taken as in Fig. \ref{fig:3Dk2}. The cross sections by the planes $x=0, y=0$ and $z=0.25$ are shown.}
        \label{fig:3Dnoisy}
\end{center}
\end{figure}
\begin{figure}[H]
\begin{center}
          \begin{subfigure}{0.32\textwidth}
                \includegraphics[width=\textwidth]{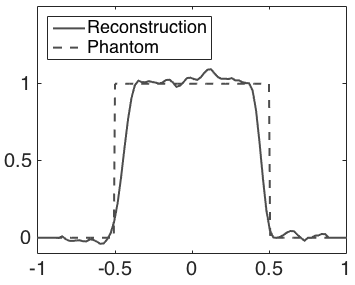}
        \end{subfigure}
  \begin{subfigure}{0.32\textwidth}
                \includegraphics[width=\textwidth]{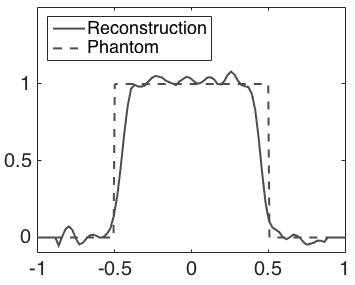}
        \end{subfigure}
  \begin{subfigure}{0.32\textwidth}
                \includegraphics[width=\textwidth]{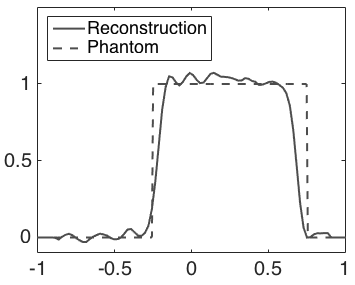}
        \end{subfigure}
        \caption{Comparison of the $x$-axis (left), $y$-axis (center) and $z$-axis (right) profiles of the reconstruction and the phantom given in Fig. \ref{fig:3Dnoisy}.}
       \label{fig:3Dnoisyprofile}
\end{center}
\end{figure}
\subsection{3D Image Reconstruction from Weighted Divergent Beam Data for k=1 and k=2}
In dimension three, when $k=0$, Theorem \ref{T:InviaRadon} reduces to Grangeat's formula \cite{Grangeat}. Here, we provide examples of reconstruction from the weighted divergent beam data for $k=1$ and $k=2$ using a spherical source geometry. For $k=1$, Theorem \ref{T:InviaRadon} gives the following inversion formula:
\begin{align}\label{DivBeamInversion3Dk1}
 f(x) = \frac{1}{8\pi^2} \int_{\S^2} \Big(\frac{\partial^2}{\partial s^2} \int_{\S^2} D^1f(s\beta+y,\sigma)\delta(\sigma \cdot \beta) d\sigma \Big)\Big|_{s=x \cdot \beta} d\beta,
\end{align}
and for $k=2$, we have
\begin{align}\label{DivBeamInversion3Dk2}
 f(x) = \frac{-1}{16\pi^2} \int_{\S^2} \Big(\frac{\partial^3}{\partial s^3} \int_{\S^2} D^2f(s\beta+y,\sigma)\text{sgn}(\sigma \cdot \beta) \Big)\Big|_{s=x \cdot \beta} d\beta.
\end{align}

The forward simulations of weighted divergent beam projections were done numerically using 1800 counts for sources $u$ on the unit sphere, 30K counts for unit directions $\sigma$. For the triangulation of the sphere, we used the algorithm given in \cite{Persson}.

Figure \ref{fig:DivBeam3Dk1} shows the three cross sections of the spherical phantom and of its reconstructions from the weighted divergent beam data obtained via \eqref{DivBeamInversion3Dk1}. The comparison of the phantom and the reconstruction given in Figure \ref{fig:DivBeam3Dk1} in terms of their coordinate axis profiles is provided in Figure \ref{fig:DivBeam3Dprofilek1}.
\begin{figure}[H]
\begin{center}
                 \begin{subfigure}{0.5\textwidth}
                \includegraphics[width=\textwidth]{Phantom3D.png}
        \end{subfigure}
\hspace{-1em}
        \begin{subfigure}{0.5\textwidth}
                \includegraphics[width=\textwidth]{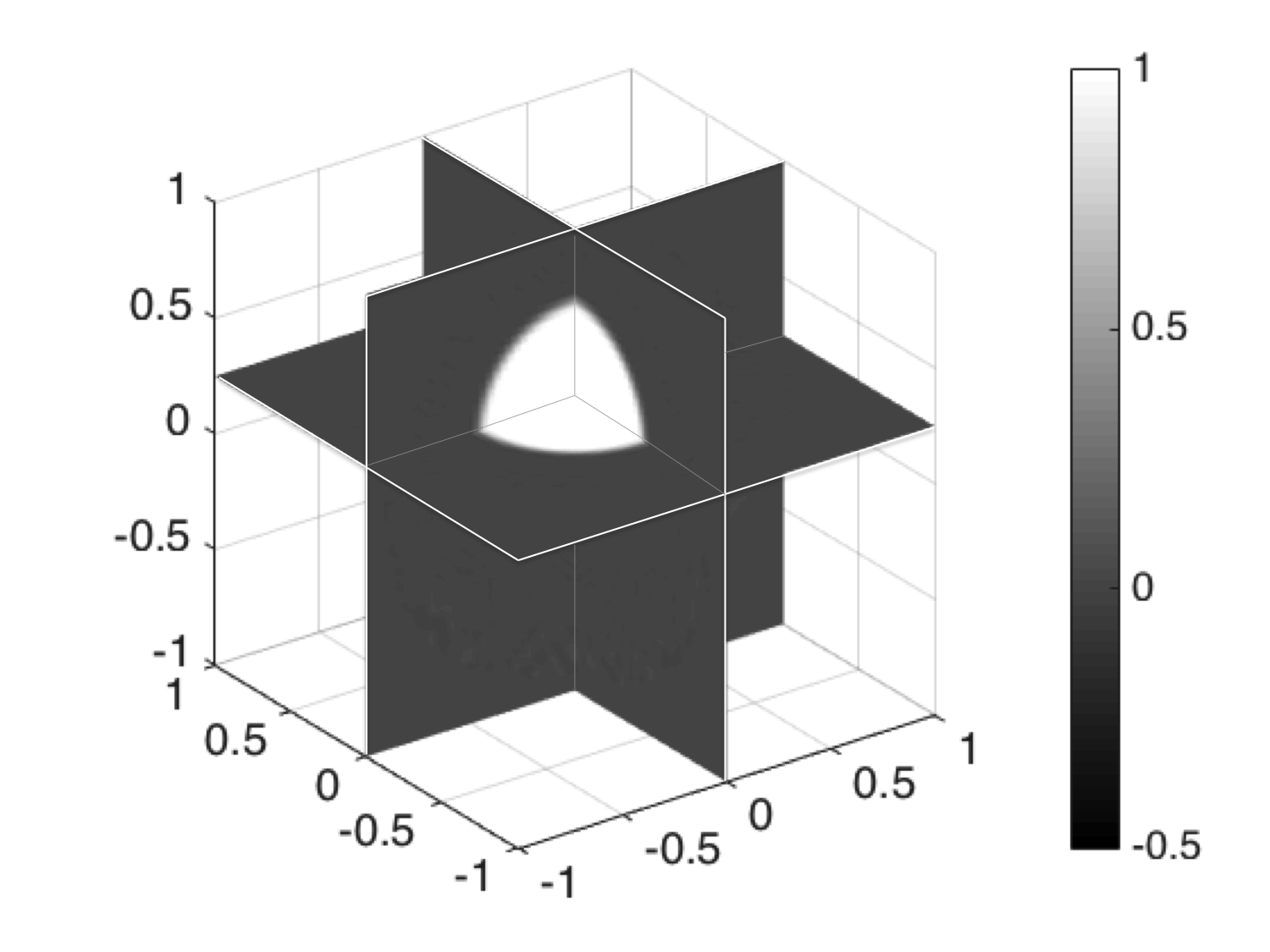}
        \end{subfigure}
       \caption{The 3D ball phantom with radius 0.5, center (0,0,0.25) and unit density (left), and $90 \times 90$ image reconstructed via \eqref{DivBeamInversion3Dk1} from weighted divergent beam data simulated using 1800 counts for sources $u$ on the unit sphere and 30K counts for unit directions $\sigma$ (right). The cross sections by the planes $x=0, y=0$ and $z=0.25$ are shown.}
        \label{fig:DivBeam3Dk1}
\end{center}
\end{figure}
\begin{figure}[H]
\begin{center}
             \begin{subfigure}{0.32\textwidth}
                \includegraphics[width=\textwidth]{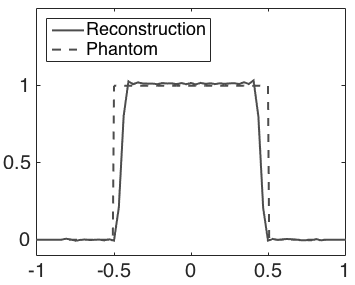}
        \end{subfigure}
  \begin{subfigure}{0.32\textwidth}
                \includegraphics[width=\textwidth]{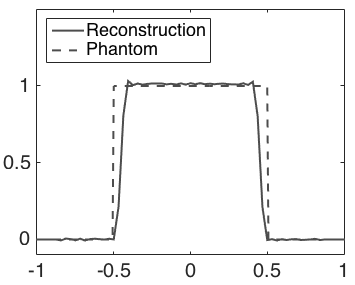}
        \end{subfigure}
  \begin{subfigure}{0.32\textwidth}
                \includegraphics[width=\textwidth]{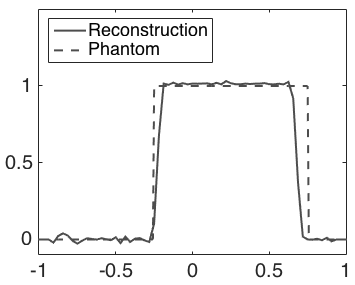}
        \end{subfigure}
        \caption{Comparison of the $x$-axis (left), $y$-axis (center) and $z$-axis (right) profiles of the phantom and the reconstruction given in Fig. \ref{fig:DivBeam3Dk1}.}
       \label{fig:DivBeam3Dprofilek1}
\end{center}
\end{figure}
Figure \ref{fig:DivBeam3Dk2} shows the three cross sections of the spherical phantom and of its reconstructions from the weighted divergent beam data obtained via \eqref{DivBeamInversion3Dk2}. The comparison of the phantom and the reconstruction given in Figure \ref{fig:DivBeam3Dk2} in terms of their coordinate axis profiles is provided in Figure \ref{fig:DivBeam3Dprofilek2}.
\begin{figure}[H]
\begin{center}
                 \begin{subfigure}{0.5\textwidth}
                \includegraphics[width=\textwidth]{Phantom3D.png}
        \end{subfigure}
\hspace{-1em}
        \begin{subfigure}{0.5\textwidth}
                \includegraphics[width=\textwidth]{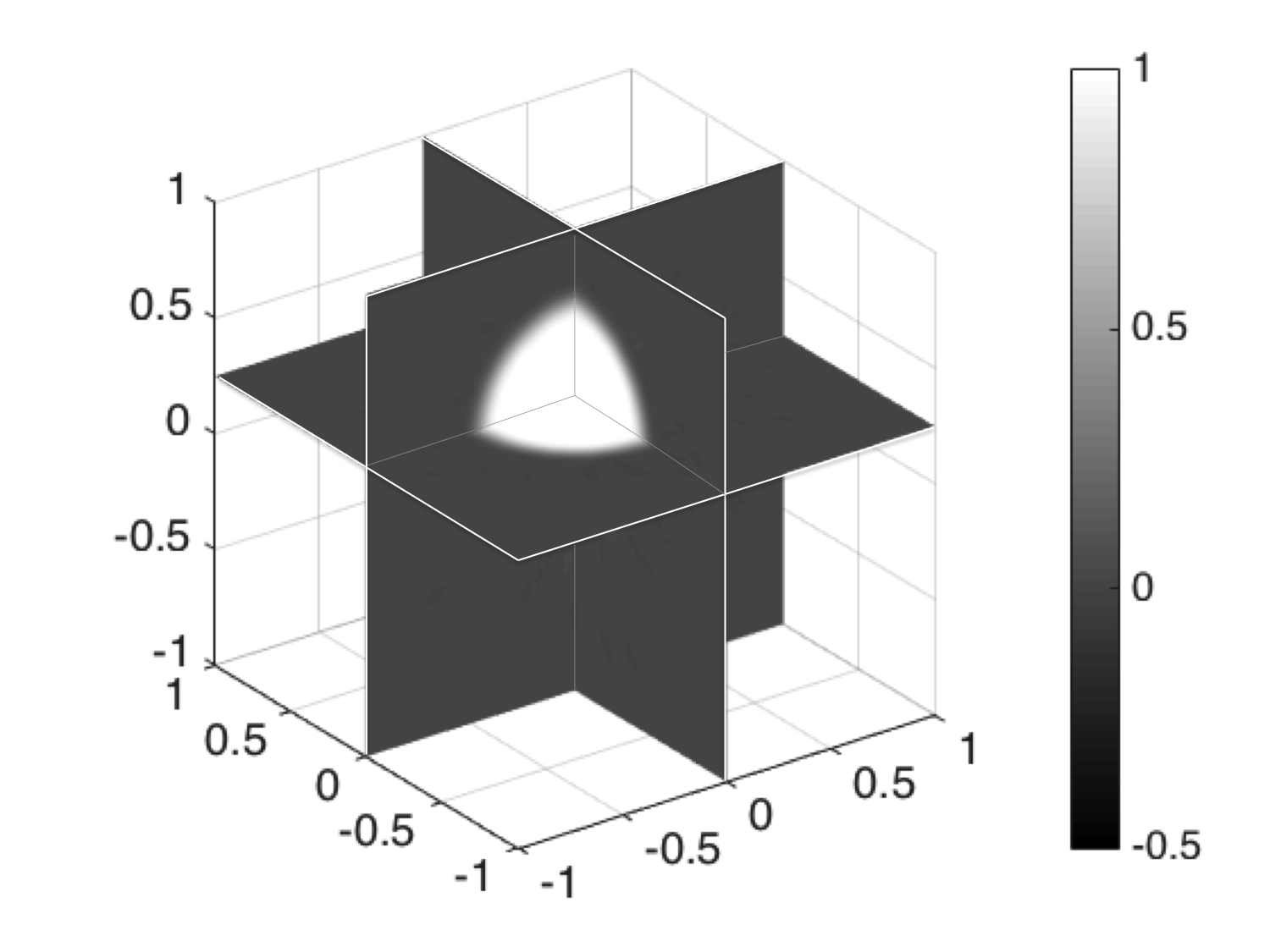}
        \end{subfigure}
       \caption{The 3D ball phantom with radius 0.5, center (0,0,0.25) and unit density (left), and $90 \times 90$ image reconstructed via \eqref{DivBeamInversion3Dk2} from weighted divergent beam data simulated using 1800 counts for sources $u$ on the unit sphere and 30K counts for unit directions $\sigma$ (right). The cross sections by the planes $x=0, y=0$ and $z=0.25$ are shown.}
        \label{fig:DivBeam3Dk2}
\end{center}
\end{figure}
\begin{figure}[H]
\begin{center}
             \begin{subfigure}{0.32\textwidth}
                \includegraphics[width=\textwidth]{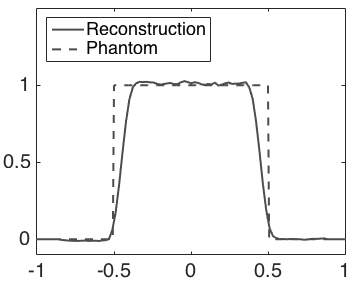}
        \end{subfigure}
  \begin{subfigure}{0.32\textwidth}
                \includegraphics[width=\textwidth]{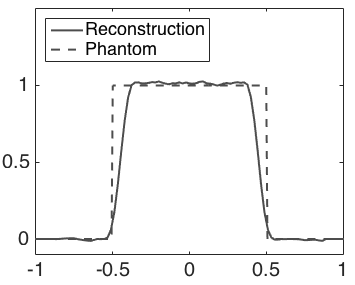}
        \end{subfigure}
  \begin{subfigure}{0.32\textwidth}
                \includegraphics[width=\textwidth]{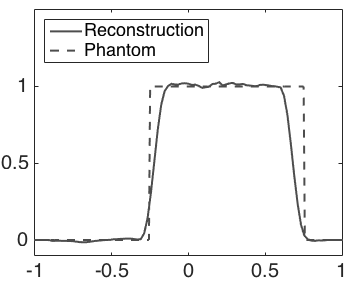}
        \end{subfigure}
        \caption{Comparison of the $x$-axis (left), $y$-axis (center) and $z$-axis (right) profiles of the reconstruction and the phantom given in Fig. \ref{fig:DivBeam3Dk2}.}
       \label{fig:DivBeam3Dprofilek2}
\end{center}
\end{figure}
\section{Conclusions and Remarks}\label{S:remarks}
\begin{itemize}
\item In this paper, we present several novel inversion formulas for the weighted divergent beam and cone transforms. In most combinations of the weights and dimensions, such formulas have not been known before. Even in the non-weighted case, the formulas are different from developed previously, in particular in \cite{Terzioglu,KuchTer}. The main (but not the only one) trigger for such studies is applicability to Compton camera imaging.
\item One of the most important features, in the authors' view, is that the new formulas are adjustable to a wide variety of (detector) geometries. We introduce the class of such geometries satisfying what we call in the text the \emph{Tuy's condition} (its weaker form was called \emph{admissibility} in \cite{KuchTer,Terzioglu}). Most of the previously derived formulas required very symmetric geometries, allowing for harmonic analysis tools to be used.
\item The latter remark is related to the important issue of understanding the geometries that allow for (stable) reconstruction. They deserve a much more thorough study, which we plan to address in another publication.
\item As it was mentioned in the introduction, to avoid being distracted from the main purpose of this text, we assume that the functions to be reconstructed belong to the Schwartz space $\mathcal{S}$. However, as in the case of Radon transform (see, e.g. \cite{Natt_old,KuchCBMS}), the results undoubtedly have a much wider area of applicability, since the derived formulas can be extended by continuity to some wider function spaces. Although we do not do this in the current text, this conclusion is confirmed, in particular, by our successful numerical implementations for discontinuous (piecewise continuous) phantoms. The issues of appropriate function spaces will be addressed elsewhere.
\item Practical soundness of the derived inversion techniques is shown by their numerical implementation in the most interesting dimensions two and three. One should also notice, that the new algorithm of the $3D$ cone transform inversion works much faster than some of the ones developed in \cite{KuchTer}. The reason is that a much coarser mesh (1.8K nodes) on the sphere suffices, rather than $30K$ used in \cite{KuchTer}.
\end{itemize}

\section*{Acknowledgments} The authors thank the NSF for the support and Numerical Analysis and Scientific Computing group of the Department of Mathematics at Texas A$\&$M University for letting the authors use their computing resources.

\medskip
Received xxxx 20xx; revised xxxx 20xx.
\medskip

\end{document}